\documentclass[letterpaper, 11pt,  reqno]{amsart}
\usepackage[margin=1.2in,marginparwidth=1.5cm, marginparsep=0.5cm]{geometry}
\usepackage[hidelinks]{hyperref}
\usepackage{amsmath,amssymb,amscd,amsthm,amsxtra, esint, xcolor,mathtools,enumerate}
\usepackage[most]{tcolorbox}
\usepackage{cite}
\usepackage{bbm}
\usepackage{enumitem}
\allowdisplaybreaks[2]

\newcommand{\Cov}{\mathrm{Cov}}

\usepackage{color}

\newtheorem{theorem}{Theorem} [section]

\newtheorem{lemma}[theorem]{Lemma}
\newtheorem{proposition}[theorem]{Proposition}
\newtheorem{remark}[theorem]{Remark}

\newtheorem{definition}[theorem]{Definition}
\newtheorem{corollary}[theorem]{Corollary}

\newcommand{\R}{\mathbb{R}}

\numberwithin{equation}{section}
\numberwithin{theorem}{section}

\makeatletter
\@namedef{subjclassname@2020}{\textup{2020} Mathematics Subject Classification}
\makeatother

\begin{document}
\baselineskip = 14pt

\title[Gaussian fluctuation for spatial average]
{
Gaussian fluctuation for spatial average of the space--time fractional stochastic heat equation
}

\author[Y. Li, Y. Hu and L. Yan]
{Yongkang Li, Yaozhong Hu$^*$ and Litan Yan}

\thanks{Y. Li was supported by the Fundamental Research Funds for the Central Universities of China (CUSF-DH-T-2024066). Y. Hu was supported by NSERC Discovery grant RGPIN 2024-05941 and a centennial fund of University of Alberta. L. Yan was supported by the National Natural Science Foundation of China (No. 11971101) and the Natural Science Foundation of Shanghai Municipality (No. 24ZR1402900).}

\address{
Yongkang Li\\
Donghua University, 2999 North Renmin Rd., Songjiang, Shanghai 201620, P.R. China}
\address{Department of Mathematical and statistical sciences, University of Alberta, Edmonton, AB, T6G 2G1, Canada}

\email{yongkangli@mail.dhu.edu.cn}

\address{
Yaozhong Hu\\
Department of Mathematical and Statistical Sciences, University of Alberta at Edmonton, Edmonton, Canada, T6G 2G1
}

\email{yaozhong@ualberta.ca}

\address{
Litan Yan\\
Donghua University, 2999 North Renmin Rd., Songjiang, Shanghai 201620, P.R. China
}

\email{litan-yan@hotmail.com}



\begin{abstract}  
	We study spatial averages of the mild solution to a one-dimensional space--time fractional stochastic heat equation driven by space--time white noise. For fixed \(t>0\), we prove a quantitative central limit theorem for the normalized spatial average over \([-R,R]\):
as \(R\to\infty\), its law converges to the standard normal law at
rate \(R^{-1/2}\) in total variation distance. The proof relies on the Malliavin--Stein method,  combined with precise estimates for the space--time fractional heat kernel and for the Malliavin derivative of the mild solution. We further establish a functional central limit theorem.
\end{abstract}

\date{\today}

\maketitle

\baselineskip = 14pt

\section{Introduction}
In this paper, we consider the following one-dimensional space-time fractional stochastic heat equation
\begin{equation}\label{eq:main equation}
	\left\{
	\begin{aligned}
		&\partial_t^\beta u(t,x)=-\nu (-\Delta)^{\frac{\alpha}{2}} u(t,x)+I_t^{1-\beta}\left[\sigma(u(t,x))\dot{W}(t,x)\right],\quad t\geq0,x\in\mathbb{R}\\
		&u_0(x)=1, x\in\R
	\end{aligned}
	\right.
\end{equation}
where \(\dot W\) denotes space--time white noise, \(\nu>0\),
\(\beta\in(0,1)\), and \(\alpha\in(1,2]\).
The coefficient \(\sigma\) is assumed to be Lipschitz continuous; that is,
there exists a constant \(L>0\) such that, for all \(x,y\in\mathbb R\),
\begin{equation}\label{eq:sigma-condition}
	|\sigma(x)-\sigma(y)|\le L|x-y|.
\end{equation}
Moreover, since the initial condition is \(u_0\equiv1\), we assume
\[
\sigma(1)\neq0,
\]
which excludes the trivial case and will be used to guarantee the
non-degeneracy of the limiting variance. 
The operator $ -(-\Delta)^{\frac{\alpha}{2}} $ denotes the fractional Laplacian operator, the generator of an isotropic stable process. The fractional differential operator is understood in the Caputo sense (see Caputo~\cite{caputo1967linear}). It is given by
\[\partial_t^\beta u(t,x)=\frac{1}{\Gamma(1-\beta)}\int_0^t\frac{\partial_ru(r,x)}{(t-r)^\beta}dr,\]
and $ I_t^{1-\beta} $ denotes the fractional integral
\[I_t^{1-\beta}u(t,x)=\frac{1}{\Gamma(1-\beta)}\int_0^t(t-r)^{-\beta}u(r,x)dr,\quad\beta\in(0,1).\]

Spatial averages of random-field solutions to stochastic partial differential equations (SPDEs) provide a natural means of describing their macroscopic fluctuation behavior. When the initial condition and the driving noise are spatially homogeneous, the solution is typically stationary in space at fixed time. Accordingly, for a growing domain \(B_R\), one is led to study \[ \int_{B_R} \bigl( u(t,x)-\mathbb E[u(t,x)] \bigr)\,dx, \] whose first-order behavior is related to spatial ergodicity, while its second-order fluctuations are described by central limit theorems and functional limit theorems.

Quantitative central limit theorems for spatial averages of stochastic heat equations were initiated by Huang \emph{et al.}~\cite{huang2020} for the one-dimensional equation driven by space--time white noise. By combining Malliavin calculus with Stein's method, they obtained an explicit normal-approximation bound in total variation
distance. This approach was subsequently extended to spatially colored Gaussian noise by Huang \emph{et al.}~\cite{huang2020color}, where both the normalization and the rate depend on the spatial covariance structure. For stochastic fractional heat equations involving a fractional Laplacian, Assaad \emph{et al.}~ \cite{assaad2022fra} established quantitative normal approximations under several classes of Gaussian multiplicative noise, including space--time white noise, Riesz-type covariance, and bounded integrable spatial covariance.

Related limit theorems have also been developed for occupation fields, more general nonlinear functionals, and singular initial data. Chen \emph{et al.}~\cite{chen2022central} proved functional central limit theorems for occupation fields of nonlinear stochastic heat equations with spatially homogeneous Gaussian noise. In a subsequent work, Chen \emph{et al.}~ \cite{chen2023central} obtained quantitative and functional central limit theorems for spatial averages of \(g(u)\), allowing both globally Lipschitz functions and certain locally Lipschitz functions. For the parabolic Anderson model with delta initial condition, Chen \emph{et al.}~ \cite{chen2022spatial} introduced a renormalized stationary field and proved spatial ergodicity and quantitative central limit theorems. These results were extended to higher dimensions and more general spatial covariance structures by Khoshnevisan \emph{et al.}~\cite{khoshnevisan2021spatial}, while Zhang \emph{et al.}~\cite{zhang2025functional} established corresponding functional central limit theorems.

The fluctuation behavior is also sensitive to the temporal structure and the distributional nature of the noise. Nualart \emph{et al.}~ \cite{nualart2021spatial,nualart2022quantitative} studied the parabolic Anderson model driven by Gaussian noise colored in both time and space. Since temporal correlation destroys the martingale structure of the time-white setting, their arguments rely on Wiener chaos expansions, Feynman--Kac representations, fourth-moment methods, and modified Gaussian Poincar\'e inequalities. For time-independent noise, Balan and Yuan~ \cite{balan2023central} obtained quantitative and functional central limit theorems covering integrable covariance functions, Riesz kernels, and one-dimensional rough fractional noise. Beyond Gaussian models, Dhoyer and Tudor~\cite{dhoyer2022spatial} derived a non-Gaussian Rosenblatt limit for a stochastic heat equation driven by Rosenblatt noise, whereas Li and Pu~ \cite{li2023gaussian} proved convergence of spatial averages of the density of one-dimensional super-Brownian motion to a Brownian sheet.

Further extensions include stochastic heat equations on bounded domains and over time-dependent averaging regions, as studied respectively by Pu~ \cite{pu2022gaussian} and Kim and Yi~\cite{kim2022limit}. Quantitative and functional central limit theorems have also been obtained for stochastic pseudo-partial differential equations with fractional noise by Liu and Shen~ \cite{liu2023gaussian}, and for nonlinear stochastic heat equations with an additional drift term by Balan and Salins~\cite{balan2025gaussian2}. Related Gaussian fluctuation results for stochastic wave equations can be found in \cite{bolanos2021averaging,nualart2022central,balan2022hyperbolic, balan2025gaussian} and the references therein.

Despite this extensive literature, the space--time fractional stochastic heat equation considered here presents a distinct difficulty. Its fundamental solution has a time-fractional scaling structure and, unlike the classical heat kernel or the stable heat kernel, does not satisfy the convolution semigroup identity 
\[ G_{t+s}=G_t\ast G_s, \qquad t,s>0. 
\] 
Consequently, several standard arguments used in the analysis of spatial averages cannot be applied directly. The key point of the present work is to replace the missing semigroup identity by estimates obtained through comparison with the stable heat kernel. This yields a quasi-semigroup bound sufficient for the covariance analysis and for the estimate of the Malliavin derivative.

Our proofs are based on the Malliavin--Stein method; see, for instance,
Nourdin and Peccati~\cite{nourdin2012normal}. If \(F=\delta(v)\) is a centered random variable with unit variance, then \[ d_{\mathrm{TV}}(F,Z) \leq 2\sqrt{ \operatorname{Var} \left( \langle DF,v\rangle_{\mathcal H} \right) }, \] where \(Z\sim N(0,1)\). Thus, quantitative normal approximation is reduced to estimating the variance of a Malliavin inner product. In the present setting, this requires precise covariance estimates for the nonlinear coefficient and sharp bounds for the Malliavin derivative of the solution.

The main contributions of this paper are as follows. First, we identify the asymptotic covariance structure of the centered spatial averages, which determines both the fluctuation scale and the covariance of the limiting process. Second, for each fixed \(t>0\), we prove a quantitative central limit theorem in total variation distance with rate \(R^{-1/2}\). Third, we establish a functional central limit theorem in \(C([0,T])\), with limiting process \[ \left( \int_0^t\sqrt{2\xi(s)}\,dB_s \right)_{t\in[0,T]}, \qquad \xi(s)=\mathbb E[\sigma(u(s,0))^2]. \]

Following Walsh's stochastic integration framework~\cite{Walsh1986}
and Dalang's extension~\cite{dalang1999extending}, we use the following notion
of a mild solution.
\begin{definition}
	Let $(\Omega, \mathcal{F}, \{\mathcal{F}_t\}_{t\geq 0}, \mathbb{P})$ be a filtered probability space satisfying the usual conditions, and let $W(ds, dy)$ be a space-time white noise defined on this space.  An adapted and jointly measurable random field
	\[
	\{u(t,x):(t,x)\in[0,T]\times\mathbb R\}
	\]
	is called a mild solution to \eqref{eq:main equation} if
	\[
	\sup_{(t,x)\in[0,T]\times\mathbb R}
	\mathbb E|u(t,x)|^2<\infty,
	\]
	and, for every
	\((t,x)\in[0,T]\times\mathbb R\),
	\[
	\int_0^t\int_{\mathbb R}
	G_{t-s}(x-y)^2
	\mathbb E\left[|\sigma(u(s,y))|^2\right]\,dy\,ds
	<\infty,
	\]
	so that the Walsh integral below is well defined in \(L^2(\Omega)\). Moreover,
	for every \((t,x)\in[0,T]\times\mathbb R\),
	\begin{equation}\label{eq:solution}
		u(t,x)
		=
		1+
		\int_0^t\int_{\mathbb R}
		G_{t-s}(x-y)\sigma(u(s,y))\,W(ds,dy)
	\end{equation}
	holds in \(L^2(\Omega)\). 
\end{definition}

Under the Lipschitz condition \eqref{eq:sigma-condition}, the existence and
uniqueness of the mild solution follow from the Picard iteration argument; see
Mijena and Nane~\cite{mijena2015space}. Moreover, for every \(p\ge2\) and
\(T>0\), there exists a constant \(C_{p,T}>0\) depending on $p$ and $T$ such that
\begin{equation}\label{eq:solution-m-bound}
	\sup_{(t,x)\in[0,T]\times\mathbb R}
	\mathbb E|u(t,x)|^p
	\le
	C_{p,T}.
\end{equation}
Since the initial condition is constant and the kernel \(G_t\) is translation-invariant in space, the solution is spatially stationary.

 In this paper, we are more interested in the asymptotic behavior as $ R $ tends to infinity of the normalized quantity
\begin{equation}\label{eq:spatial averager quantity}
	F_R(t):=\frac{1}{\sigma_R}\left(\int_{-R}^R u(t,x)dx-\mathbb{E}\left[\int_{-R}^R u(t,x)dx\right]\right),
\end{equation}
with $ R>0 $ and 
\begin{equation}\label{eq:sigma_R}
	\sigma_R^2:=\mathrm{Var}\left(\int_{-R}^Ru(t,x)dx\right) .
\end{equation}

A key step is to identify the asymptotic behavior of the
normalizing constant. More precisely, in Proposition \ref{pro:cov_GR}, we will prove that
\begin{equation}\label{eq:est_sigma_R}
	\frac{\sigma_R^2}{R}
	\longrightarrow
	2\int_0^t\xi(r)\,dr,
	\qquad R\to\infty,
\end{equation}
where $\xi(r)$ denotes $\mathbb{E}[\sigma(u(r,x))^2]$.
Equivalently,
\[
\sigma_R^2
\sim
2R\int_0^t\xi(r)\,dr.
\]

We now state the two main results. The first is a fixed-time quantitative
central limit theorem with an explicit total variation rate.
\begin{theorem}\label{QCLT}
	Under the assumptions \eqref{hyp:main}, let \(d_{\mathrm{TV}}\) denote the total variation distance and let
	\(Z\sim N(0,1)\). Then, for every fixed \(t>0\), there exists a positive
	constant \(C\), independent of $R$, such that
	\[
	d_{\mathrm{TV}}(F_R(t),Z)\le \frac{C}{\sqrt R},\qquad R>0.
	\]
\end{theorem}
We also provide a functional version of Theorem \ref{QCLT}.
\begin{theorem}\label{FCLT}
	Under the assumptions \eqref{hyp:main}, let $\{u(t,x), (t,x)\in[0,T]\times\R\}$ be the mild solution to the SPDE~\eqref{eq:main equation} given by \eqref{eq:solution}. Denote
	\[
	\xi(s)=\mathbb{E}\left[\sigma(u(s,y))^2\right], \qquad s\geq 0.
	\]
	Then, for any $T>0$,
	\[
	\left(
	\frac{1}{\sqrt{R}}
	\left(
	\int_{-R}^{R} u(t,x)\,dx - 2R
	\right)
	\right)_{t\in[0,T]}
	\longrightarrow
	\left(
	\int_0^t \sqrt{2\xi(s)}\,dB_s
	\right)_{t\in[0,T]},
	\]
	 as $R$ tends to infinity, where $B$ is a standard Brownian motion and the convergence is in law on the space of continuous functions $C([0,T])$.
\end{theorem}

The rest of the paper is organized as follows. In Section~\ref{2}, we recall
some preliminaries on space--time white noise, the space--time fractional heat
kernel, Malliavin calculus and Stein's method. In Section~\ref{3}, we study the
asymptotic behavior of the covariance of the spatial averages of the mild
solution. Section~\ref{4} is devoted to the proof of the quantitative central
limit theorem, namely Theorem~\ref{QCLT}, while Section~\ref{5} is devoted to
the proof of the functional central limit theorem, namely Theorem~\ref{FCLT}.
Finally, in Appendix~\ref{6}, we provide the \(L^p\)-estimate, for \(p\ge2\),
of the Malliavin derivative of the solution to SPDE~\eqref{eq:main equation}.

\section{Preliminaries}\label{2}
In this section, we collect several preliminary facts that will be used
throughout the paper. We first recall the notation for space--time white
noise and the associated Walsh stochastic integral. We then summarize some
properties of the space--time fractional heat kernel. Finally, we recall the
basic elements of Malliavin calculus and the Malliavin--Stein bound used in the
proof of the quantitative central limit theorem. Throughout the paper, \(C\) denotes a generic positive constant whose value may
change from line to line. Unless otherwise specified, \(C\) may depend on the
fixed parameters of the model, but is independent of the spatial scale \(R\).

We first specify the meaning of the space--time white noise appearing in
\eqref{eq:main equation}. The noise \(\dot W\) is understood as a centered
Gaussian generalized random field on \(\mathbb R_+\times\mathbb R\) with formal
covariance
\[
\mathbb E[\dot W(t,x)\dot W(s,y)]
=
\delta_0(t-s)\delta_0(x-y),
\qquad t,s\ge0,\ x,y\in\mathbb R.
\]
More precisely, for deterministic functions
\(\varphi,\psi\in L^2(\mathbb R_+\times\mathbb R)\), set
\[
W(\varphi)
:=
\int_0^\infty\int_{\mathbb R}
\varphi(t,x)\,W(dt,dx).
\]
Then \(W\) is a centered Gaussian family satisfying
\[
\mathbb E[W(\varphi)W(\psi)]
=
\int_0^\infty\int_{\mathbb R}
\varphi(t,x)\psi(t,x)\,dx\,dt.
\] Throughout this paper, stochastic integrals with
respect to \(W\) are understood in the sense of Walsh. In particular, if
\(X=\{X(s,y):(s,y)\in[0,T]\times\mathbb R\}\) is an adapted and jointly
measurable random field satisfying
\[
\int_0^T\int_{\mathbb R}
\mathbb E|X(s,y)|^2\,dy\,ds<\infty,
\]
then
\[
\int_0^T\int_{\mathbb R}X(s,y)\,W(ds,dy)
\]
is well defined in \(L^2(\Omega)\), and the Itô isometry holds:
\[
\mathbb E\left[
\left(
\int_0^T\int_{\mathbb R}X(s,y)\,W(ds,dy)
\right)^2
\right]
=
\int_0^T\int_{\mathbb R}
\mathbb E|X(s,y)|^2\,dy\,ds.
\]
\subsection{Some properties of the space--time fractional heat kernel}
In this subsection, we collect several properties of the space--time fractional heat kernel that will be used throughout the paper. We first present the kernel in a general spatial
	dimension \(d\), so as to clarify the origin of the dimensional restriction.
	Afterwards, we specialize to the one-dimensional setting considered in this
	paper. 
	
	Let \(G_t(x)\), \(t>0\), \(x\in\mathbb R^d\), be the fundamental solution
	to the deterministic fractional heat equation
	\begin{equation}\label{eq:PDE}
		\begin{cases}
			\partial_t^\beta G_t(x)
			=
			-\nu(-\Delta)^{\alpha/2}G_t(x),
			& t>0,\ x\in\mathbb R^d,\\
			G_0(x)=\delta_0(x),&x\in\R^d.
		\end{cases}
	\end{equation}
	 Let $X=\{X(t), t\geq 0\}$ be an isotropic $\alpha$-stable L\'evy process in $\mathbb{R}^d$ with density function denoted by $p_\alpha(t,x)$ whose Fourier transform can be expressed as
 \[ \widehat p_\alpha(t,\xi)=e^{-\nu t|\xi|^\alpha}\qquad\xi\in\R^d. \]
  See Chen and Dalang\cite{chen2015moments} and Debbi and Dozzi\cite{debbi2005solutions} for some properties of $p_\alpha$. Let \[ E_t:=\inf\{r>0:D_r>t\} \] be the first passage time of a \(\beta\)-stable subordinator \(D=\{D_r:r\ge0\}\). (See Meerschaert and Scheffler \cite{meerschaert2004limit} for properties of the inverse stable subordinator.) It is known that \(G_t(x)\) is the density of the time-changed process \(X(E_t)\). Hence, by conditioning on \(E_t\), \begin{equation}\label{eq:G=integral} G_t(x) = \int_0^\infty p_\alpha(s,x)f_{E_t}(s)\,ds, 
\end{equation}
where 
\begin{equation}\label{eq:f-G-integral}
	f_{E_t}(s) = t\beta^{-1}s^{-1-1/\beta} g_\beta\left(ts^{-1/\beta}\right), \qquad s>0,
\end{equation}
is the density of $E_t$ and \(g_\beta\) is the density of \(D_1\). Equivalently, after the change of variables \(\tilde{s}=ts^{-1/\beta}\), one has \begin{equation}\label{eq:G-subordination-change} G_t(x) = \int_0^\infty p_\alpha\left((t/s)^\beta,x\right)g_\beta(s)\,ds. 
\end{equation}
We will also need the Fourier transform of $G_t(x)$:
\begin{equation}\label{eq:G-Fourier}
	\widehat{G}_t(\xi)=E_\beta(-\nu|\xi|^\alpha t^\beta),
\end{equation}
where 
\[E_\beta(x):=\sum_{k=0}^\infty \frac{x^k}{\Gamma(1+\beta k)}\]
denotes the Mittag--Leffler function, and \(\Gamma\) denotes the
Gamma function. The function \(x\mapsto E_\beta(-x)\) is completely
monotone on \([0,\infty)\). Moreover, the Mittag-Leffler function has the uniform estimate
\[
\frac{1}{1+\Gamma(1-\beta)x}
\leq
E_\beta(-x)
\leq
\frac{1}{1+\Gamma(1+\beta)^{-1}x},
\qquad x>0.
\]

In contrast to the stable heat kernel \(p_\alpha\), the family
\(\{G_t\}_{t>0}\) is not, in general, a convolution semigroup when
\(0<\beta<1\). More precisely,
\[
G_{t+s}\neq G_t*G_s,
\qquad s,t>0.
\]
Indeed, if \(\{G_t\}_{t>0}\) were a convolution semigroup, then
\[
\widehat G_{t+s}(\xi)
=
\widehat G_t(\xi)\widehat G_s(\xi).
\]
Using the Fourier representation \eqref{eq:G-Fourier}, this would imply
\[
E_\beta\left(-\nu|\xi|^\alpha(t+s)^\beta\right)
=
E_\beta\left(-\nu|\xi|^\alpha t^\beta\right)
E_\beta\left(-\nu|\xi|^\alpha s^\beta\right).
\]
Let \(a=\nu|\xi|^\alpha\). As \(a\downarrow0\), the expansion of the
Mittag--Leffler function gives
\[
E_\beta(-a(t+s)^\beta)
=
1-\frac{a(t+s)^\beta}{\Gamma(1+\beta)}
+O(a^2),
\]
whereas
\[
E_\beta(-at^\beta)E_\beta(-as^\beta)
=
1-\frac{a(t^\beta+s^\beta)}{\Gamma(1+\beta)}
+O(a^2).
\]
Since
\[
(t+s)^\beta\neq t^\beta+s^\beta
\]
in general for \(0<\beta<1\), the convolution semigroup property fails.

The following three estimates and properties for the kernel $G_t$ play a crucial role in our subsequent proofs.
\begin{lemma}[Lemma 1 in \cite{mijena2015space}]\label{le:G^2}
	For $ d<2\alpha $,
	\[\int_{\mathbb{R}^d}G_t^2(x)dx=C^*t^{-\beta d/\alpha},\]
	where $$ C^*=\frac{(\nu)^{-d/\alpha}2\pi^{d/2}}{\alpha\Gamma\left(\frac{d}{2}\right)}\frac{1}{(2\pi)^d}\int_0^\infty z^{d/\alpha-1}(E_\beta(-z))^2dz. $$
\end{lemma}
\begin{lemma}[Lemma 2.1 in \cite{foondun2017asymptotic}]\label{le:G-sup}
	If \(\alpha>d\), then there exist two positive constants \(c_1,c_2\) such that, for all
	\(t>0\) and \(x\in\mathbb R^d\),
	\[
	c_1\left(t^{-\beta d/\alpha}\wedge \frac{t^\beta}{|x|^{d+\alpha}}\right)
	\leq G_t(x)\leq
	c_2\left(t^{-\beta d/\alpha}\wedge \frac{t^\beta}{|x|^{d+\alpha}}\right).
	\]
	Equivalently,
	\[
	G_t(x)\asymp t^{-\beta d/\alpha}\wedge \frac{t^\beta}{|x|^{d+\alpha}},
	\qquad t>0,\ x\in\mathbb R^d .
	\]
\end{lemma}
The above two estimates impose different dimensional restrictions. 
Lemma~\ref{le:G^2} requires \(d<2\alpha\), while
Lemma~\ref{le:G-sup} requires the stronger condition \(\alpha>d\). Since
\(0<\alpha\le2\), the latter condition implies that the only possible positive
integer dimension is \(d=1\), together with \(\alpha>1\). 

Throughout the paper, we work under the following standing assumption:
\begin{equation}\label{hyp:main}
	\begin{cases}
		d = 1, \quad \nu > 0, \quad \beta \in (0,1), \quad \alpha \in (1,2], \\
		\sigma \text{ is globally Lipschitz continuous with } \sigma(1) \neq 0.
	\end{cases}
\end{equation}
Moreover, the stochastic convolution driven by space--time white noise requires
\[
\int_0^t\int_{\mathbb R^d}G_{t-s}(x-y)^2\,dy\,ds<\infty,
\]
which is equivalent to \(d<\alpha/\beta\). This condition is automatically
satisfied under \(d=1\), \(\alpha\in(1,2]\), and \(0<\beta<1\).
\begin{remark}\label{rem:G-stable-comparison}
	By Lemma~\ref{le:G-sup}, in the one-dimensional case, the space--time
	fractional heat kernel satisfies
	\[
	G_t(x)
	\asymp
	t^{-\beta/\alpha}
	\wedge
	\frac{t^\beta}{|x|^{1+\alpha}},
	\qquad t>0,\ x\in\mathbb R.
	\]
	On the other hand, it is well known that
	\[
	p_\alpha(t,x)
	\asymp
	t^{-1/\alpha}
	\wedge
	\frac{t}{|x|^{1+\alpha}},
	\qquad t>0,\ x\in\mathbb R.
	\]
	Hence we get
	\[
	p_\alpha(t^\beta,x)
	\asymp
	t^{-\beta/\alpha}
	\wedge
	\frac{t^\beta}{|x|^{1+\alpha}},
	\qquad t>0,\ x\in\mathbb R.
	\]
	Consequently, \(G_t\) is comparable to the stable heat kernel evaluated at
	the operational time \(t^\beta\); namely,
	\begin{equation}\label{eq:G-stable-comparison}
		G_t(x)\asymp p_\alpha(t^\beta,x),
		\qquad t>0,\ x\in\mathbb R.
	\end{equation}
\end{remark}

\begin{lemma}\label{lem:quasi-semigroup-G}
	Suppose that \(d=1\). There exists a constant \(C>0\) such that, for all
	\(t,s>0\) and \(h\in\mathbb R\),
	\begin{equation}\label{eq:quasi-semigroup-G}
		\int_{\mathbb R}G_t(h-\xi)G_s(\xi)\,d\xi
		\le
		CG_{t+s}(h).
	\end{equation}
\end{lemma}

\begin{proof}
	By the two-sided comparison estimate \eqref{eq:G-stable-comparison}, there exists a
	constant \(C>0\) such that
	\[
	G_t(x)\le Cp_\alpha(t^\beta,x),
	\qquad t>0,\ x\in\mathbb R.
	\]
	Therefore, using the semigroup property of the stable heat kernel
	\(p_\alpha\), we obtain
	\[
	\begin{aligned}
		\int_{\mathbb R}G_t(h-\xi)G_s(\xi)\,d\xi
		&\le
		C\int_{\mathbb R}
		p_\alpha(t^\beta,h-\xi)p_\alpha(s^\beta,\xi)\,d\xi  \\
		&=
		Cp_\alpha(t^\beta+s^\beta,h).
	\end{aligned}
	\]
	Since \(0<\beta<1\), we have
	\[
	(t+s)^\beta
	\le
	t^\beta+s^\beta
	\le
	2^{1-\beta}(t+s)^\beta.
	\]
	Hence \(t^\beta+s^\beta\asymp (t+s)^\beta\). By the standard two-sided
	estimate for \(p_\alpha\), it follows that
	\[
	p_\alpha(t^\beta+s^\beta,h)
	\le
	Cp_\alpha((t+s)^\beta,h).
	\]
	Using the reverse inequality in
\eqref{eq:G-stable-comparison}, namely
\[
p_\alpha((t+s)^\beta,h)\leq CG_{t+s}(h),
\]
we obtain
	\[
	\int_{\mathbb R}G_t(h-\xi)G_s(\xi)\,d\xi
	\le
	CG_{t+s}(h).
	\]
	The proof is complete.
\end{proof}
\begin{corollary}\label{cor:G2G}
	Suppose that \(d=1\). There exists a constant \(C>0\) such that, for all
	\(t,s>0\) and \(h\in\mathbb R\),
	\begin{equation}\label{eq:G2G}
		\int_{\mathbb R}G_t(h-\xi)^2G_s(\xi)\,d\xi
		\le
		Ct^{-\beta/\alpha}G_{t+s}(h).
	\end{equation}
	Similarly,
	\begin{equation}\label{eq:GG2}
		\int_{\mathbb R}G_t(h-\xi)G_s(\xi)^2\,d\xi
		\le
		Cs^{-\beta/\alpha}G_{t+s}(h).
	\end{equation}
\end{corollary}

\begin{proof}
	By Lemma \ref{le:G-sup}, we have 
	\[
	G_t(x)\le Ct^{-\beta/\alpha},
	\qquad t>0,\ x\in\mathbb R.
	\]
	Hence
	\[
	G_t(x)^2
	\le
	Ct^{-\beta/\alpha}G_t(x).
	\]
	Therefore, by Lemma~\ref{lem:quasi-semigroup-G},
	\[
	\begin{aligned}
		\int_{\mathbb R}G_t(h-\xi)^2G_s(\xi)\,d\xi
		&\le
		Ct^{-\beta/\alpha}
		\int_{\mathbb R}G_t(h-\xi)G_s(\xi)\,d\xi  \\
		&\le
		Ct^{-\beta/\alpha}G_{t+s}(h).
	\end{aligned}
	\]
	This proves \eqref{eq:G2G}. The proof of \eqref{eq:GG2} is analogous.
\end{proof}

We now summarize some further properties of $G_t(x)$ needed in the sequel.

\begin{lemma}\label{lem:basic-G-estimates} For \(t>0\), the following properties hold. \begin{enumerate}[label=\textup{(\roman*)}, ref=\textup{(\roman*)}] 
		\item \label{it:G-property-1}
		For every \(x\in\mathbb R\), 
		\[ G_t(x)\ge0,\qquad \int_{\mathbb R}G_t(x)\,dx=1, \] 
		and 
		\[ G_t(x)=G_t(-x). \] 
		\item \label{it:G-property-2} The kernel \(G_t\) satisfies the scaling property
		\[ G_t(x) = t^{-\beta/\alpha} G_1\left(t^{-\beta/\alpha}x\right). \] 
		\item \label{it:G-property-3} There exists a constant \(C>0\) such that \[ \|G_t^{1/2}\|_{L^1(\mathbb R)} \le Ct^{\frac{\beta}{2\alpha}}. \] 
		\item \label{it:G-property-4} There exists a constant \(C>0\) such that \[ \|G_t^{1/2}*G_t^{1/2}\|_{L^1(\mathbb R)} \le Ct^{\frac \beta\alpha}. \] 
	\end{enumerate} 
\end{lemma} 
\begin{proof} 
	We prove the assertions one by one. \begin{enumerate}[label=\textup{(\roman*)}, ref=\textup{(\roman*)}] 
		\item By \eqref{eq:G=integral}, together with \(p_\alpha(s,x)\ge0\) and \(f_{E_t}(s)\ge0\), we have \(G_t(x)\ge0\). Moreover, by Fubini's theorem, \[ \int_{\mathbb R}G_t(x)\,dx = \int_0^\infty \int_{\mathbb R}p_\alpha(s,x)\,dx\,f_{E_t}(s)\,ds = \int_0^\infty f_{E_t}(s)\,ds = 1. \] Since \(p_\alpha(s,x)=p_\alpha(s,-x)\), the same representation yields \(G_t(x)=G_t(-x)\). 
		\item By \eqref{eq:G-Fourier}, 
		\[ \widehat G_t(\xi) = E_\beta\left(-\nu|\xi|^\alpha t^\beta\right). 
		\] 
		Define \[ H_t(x):=t^{-\beta/\alpha}G_1(t^{-\beta/\alpha}x). \] Then, by the change of variables \(z=t^{-\beta/\alpha}x\),
		\[
		\widehat H_t(\xi)
		=
		\widehat G_1(t^{\beta/\alpha}\xi)
		=
		E_\beta\left(-\nu|t^{\beta/\alpha}\xi|^\alpha\right)
		=
		E_\beta\left(-\nu|\xi|^\alpha t^\beta\right)
		=
		\widehat G_t(\xi).
		\]
		Hence \(H_t=G_t\), and therefore
		\[
		G_t(x)
		=
		t^{-\beta /\alpha}
		G_1(t^{-\beta/\alpha}x).
		\]
		\item  By Lemma~\ref{le:G-sup}, in the one-dimensional case, there exists a constant
		\(C>0\) such that
		\[
		G_t(x)^{1/2}
		\le
		C\left(
		t^{-\beta/(2\alpha)}
		\wedge
		t^{\beta/2}|x|^{-(1+\alpha)/2}
		\right).
		\]
		Splitting the integral into the two regions
		\[
		|x|\le t^{\beta/\alpha}
		\qquad\text{and}\qquad
		|x|>t^{\beta/\alpha},
		\]
		we obtain
		\[
		\begin{aligned}
			\int_{\mathbb R}G_t(x)^{1/2}\,dx
			&\le
			C\int_{|x|\le t^{\beta/\alpha}}t^{-\frac{\beta}{2\alpha}}\,dx
			+
			C\int_{|x|>t^{\beta/\alpha}}
			t^{\frac\beta2}|x|^{-\frac{1+\alpha}{2}}\,dx  \\
			&=
			C t^{-\frac{\beta}{2\alpha}}t^{\frac \beta\alpha}
			+
			C t^{\frac\beta2}
			\int_{|x|>t^{\beta/\alpha}}|x|^{-\frac{1+\alpha}{2}}\,dx .
		\end{aligned}
		\]
		For the second integral, since \(\alpha>1\),
		\[
			\int_{|x|>t^{\beta/\alpha}}|x|^{-\frac{1+\alpha}{2}}\,dx
			=
			2\int_{t^{\beta/\alpha}}^\infty x^{-\frac{1+\alpha}{2}}\,dx
			=
			Ct^{-\frac{\beta(\alpha-1)}{2\alpha}}.
		\]
		Therefore,
		\[
			\int_{\mathbb R}G_t(x)^{1/2}\,dx
			\le
			C t^{\frac{\beta}{2\alpha}}
			+
			C t^{\frac\beta2}
			t^{-\frac{\beta(\alpha-1)}{2\alpha}}
			=C t^{\frac{\beta}{2\alpha}}.
		\]
		\item By Young's convolution inequality and \ref{it:G-property-3}, \[ \|G_t^{1/2}*G_t^{1/2}\|_{L^1(\mathbb R)} \le \|G_t^{1/2}\|_{L^1(\mathbb R)}^2 \le Ct^{\frac\beta\alpha}. \]
		 \end{enumerate} 
\end{proof}

\subsection{Malliavin calculus}
In this subsection, we will recall some basic facts on the Malliavin calculus (for example, see Nualart~\cite{Nualart2006}). Let
\(
\mathcal H:=L^2(\mathbb R_+\times\mathbb R)
\)
with inner product
\[
\langle h,g\rangle_{\mathcal H}
:=
\int_0^\infty\int_{\mathbb R}h(s,y)g(s,y)\,dy\,ds.
\]
The space--time white noise \(W\) can be regarded as an isonormal Gaussian
process over \(\mathcal H\), that is, a centered Gaussian family
\[
\{W(h):h\in\mathcal H\}
\]
satisfying
\[
\mathbb E[W(h)W(g)]
=
\langle h,g\rangle_{\mathcal H},
\qquad h,g\in\mathcal H.
\]
For deterministic \(h\in\mathcal H\), we formally write
\[
W(h)=\int_0^\infty\int_{\mathbb R}h(s,y)\,W(ds,dy).
\]
Let \(\mathcal S\) be the class of smooth cylindrical random variables of the
form
\[F=f(W(h_1),\dots, W(h_n)),\]
where \(n\ge1\), \(h_1,\ldots,h_n\in\mathcal H\), and
\(f\in C_b^\infty(\mathbb R^n)\). Here \(C_b^\infty(\mathbb R^n)\) denotes the
space of infinitely differentiable functions on \(\mathbb R^n\) whose partial
derivatives of all orders are bounded. For \(F\in\mathcal S\), its Malliavin
derivative is defined by
\[DF=\sum_{i=1}^{n}\frac{\partial f}{\partial x_i}(W(h_1),\dots, W(h_n))h_i\]
where we denote by $D$ the Malliavin derivative operator and \(DF\) is an \(\mathcal H\)-valued random variable. For every \(p\ge1\), 
the operator \(D\) is closable from \(L^p(\Omega)\) into
\(L^p(\Omega;\mathcal H)\). We denote by \(\mathbb D^{1,p}\) the closure of
\(\mathcal S\) with respect to the norm
\[
\|F\|_{1,p}
:=
\left(
\mathbb E|F|^p
+
\mathbb E\|DF\|_{\mathcal H}^p
\right)^{1/p}.
\]
In particular, for \(p=2\),
\[
\|F\|_{1,2}^2
=
\mathbb E[F^2]
+
\mathbb E\|DF\|_{\mathcal H}^2.
\]

The divergence operator \(\delta\) is the adjoint of \(D\). Its domain,
	denoted by \(\operatorname{Dom}\delta\), consists of those random elements
	\(v\in L^2(\Omega;\mathcal H)\) such that there exists a constant \(C_v>0\)
	satisfying
	\[
	\left|\mathbb E\langle DF,v\rangle_{\mathcal H}\right|
	\le
	C_v\|F\|_{L^2(\Omega)},
	\qquad F\in\mathbb D^{1,2}.
	\]
	For \(v\in\operatorname{Dom}\delta\), \(\delta(v)\) is characterized by the
	duality relation
	\[
	\mathbb E[F\delta(v)]
	=
	\mathbb E\langle DF,v\rangle_{\mathcal H},
	\qquad F\in\mathbb D^{1,2}.
	\]
	If \(v=\{v(s,y):(s,y)\in\mathbb R_+\times\mathbb R\}\) is predictable and
	satisfies
	\[
	\mathbb E\int_0^\infty\int_{\mathbb R}|v(s,y)|^2\,dy\,ds<\infty,
	\]
	then \(v\in\operatorname{Dom}\delta\), and \(\delta(v)\) coincides with the
	Walsh integral:
	\[
	\delta(v)
	=
	\int_0^\infty\int_{\mathbb R}v(s,y)\,W(ds,dy).
	\]

The next lemma establishes the Malliavin differentiability of the mild solution
and identifies the equation satisfied by its Malliavin derivative. This result
will be used repeatedly in the proofs of  Theorem \ref{QCLT} and Theorem \ref{FCLT}.
\begin{lemma}\label{lem:Malliavin-derivative}
	Assume that $\sigma$ is Lipschitz continuous with Lipschitz
		constant \(L\). Then, for every $p\geq 2, T>0$
	and every $(t,x)\in [0,T]\times\mathbb R$, the mild solution
	$u(t,x)$ of \eqref{eq:main equation} belongs to $\mathbb D^{1,p}$. 
    
    Moreover, there exists a bounded predictable random field
		\[
		\Sigma=\{\Sigma(r,z):(r,z)\in[0,T]\times\mathbb R\},
		\qquad |\Sigma(r,z)|\le L,
		\quad
		\mathbb P\otimes dr\otimes dz\text{-a.e.},
		\]
		such that, for every \((t,x)\in[0,T]\times\mathbb R\), there exists a
		version of \(Du(t,x)\) satisfying, for almost every
		\((s,y)\in[0,t)\times\mathbb R\), \(\mathbb P\)-a.s.,
		\begin{equation}\label{eq:Malliavin-derivative}
			\begin{aligned}
				D_{s,y}u(t,x)
				&=
				G_{t-s}(x-y)\sigma(u(s,y))\\
				&\quad+
				\int_s^t\int_{\mathbb R}
				G_{t-r}(x-z)\Sigma(r,z)
				D_{s,y}u(r,z)\,W(dr,dz).
			\end{aligned}
		\end{equation}
		Moreover,
\(
D_{s,y}u(t,x)=0
\) for almost every \((s,y)\in(t,T]\times\mathbb R\),
\(\mathbb P\)-a.s. If, in addition, \(\sigma\in C^1(\mathbb R)\), then one may take
		\(\Sigma(r,z)=\sigma'(u(r,z))\).
\end{lemma}
\begin{proof}
	The proof follows the Picard iteration argument used in Proposition~2.4.4
	of Nualart~\cite{Nualart2006}.
	
	As introduced in Mijena and Nane~\cite{mijena2015space}, we consider the Picard iteration scheme defined by 
	\[u_0(t,x)=1,\]
	and for $n\geq 0$,
	\[
	u_{n+1}(t,x)
	=
	1
	+
	\int_0^t\int_{\mathbb R}
	G_{t-r}(x-z)\sigma(u_n(r,z))\,W(dr,dz).
	\] 
	The standard Picard estimates show that
\(\{u_n(t,x)\}_{n\ge0}\) is a Cauchy sequence in \(L^p(\Omega)\).
Hence,
\[
u_n(t,x)\longrightarrow u(t,x)
\quad\text{in }L^p(\Omega).
\]
	
	Moreover, the standard Picard estimates and the linear growth of \(\sigma\) yield
\begin{equation}\label{eq:uniform-Picard-moment}
\sup_{n\ge0}
\sup_{(t,x)\in[0,T]\times\mathbb R}
\left(
\|u_n(t,x)\|_p
+
\|\sigma(u_n(t,x))\|_p
\right)
<\infty.
\end{equation}

	We prove inductively that, for every \(n\ge0\), \(u_n(t,x)\in\mathbb D^{1,p}\) for all \((t,x)\in[0,T]\times\mathbb R\). Since $u_0(t,x)=1$ is deterministic, the assertion is clear for \(n=0\).
	Suppose that \(u_n(t,x)\in\mathbb D^{1,p}\) for all
\((t,x)\in[0,T]\times\mathbb R\), and that
\begin{equation}\label{eq:D_sy-assumption}
		\sup_{(t,x)\in[0,T]\times\R}
		\mathbb{E}\left[\left(
		\int_0^t \int_\R
		\left|D_{s,y}u_n(t,x)\right|^2
		\,dy\,ds
		\right)^{p/2}\right]
		<\infty.
	\end{equation}
	By the Lipschitz chain rule for Malliavin
		derivatives, 
		there exists a predictable random field \(\Sigma_n\), satisfying
		\[
		|\Sigma_n(r,z)|\le L,
		\]
		such that
		\[
		D\sigma(u_n(r,z))
		=
		\Sigma_n(r,z)Du_n(r,z).
		\]
    By Minkowski's integral inequality, \eqref{eq:uniform-Picard-moment}, and Lemma~\ref{le:G^2}, \begin{equation}\label{eq:integrand-n-Lp} \begin{aligned} 
    &\left\| \left( \int_0^t\int_{\mathbb R} |G_{t-r}(x-z)\sigma(u_n(r,z))|^2\,dz\,dr \right)^{1/2} \right\|_p^2\\ &\qquad\le \int_0^t\int_{\mathbb R} G_{t-r}(x-z)^2 \|\sigma(u_n(r,z))\|_p^2\,dz\,dr\\ &\qquad\le C\int_0^t(t-r)^{-\beta/\alpha}\,dr <\infty. 
    \end{aligned} \end{equation}
    Moreover, 
    \[ D_{s,y}\left(\mathbf 1_{[0,t]}(r) G_{t-r}(x-z)\sigma(u_n(r,z))\right) = \mathbf 1_{\{0\le s\le r\le t\}} G_{t-r}(x-z)\Sigma_n(r,z) D_{s,y}u_n(r,z). 
    \] 
    Hence, by Minkowski's integral inequality, \(|\Sigma_n|\le L\), and the induction hypothesis \eqref{eq:D_sy-assumption}, \begin{equation}\label{eq:D-integrand-n-Lp} \begin{aligned} 
    &\left\| \left( \int_0^t\int_{\mathbb R} \|D\left(\mathbf 1_{[0,t]}(r) G_{t-r}(x-z)\sigma(u_n(r,z))\right)\|_{\mathcal H}^2 \,dz\,dr \right)^{1/2} \right\|_p^2\\ &\qquad\le L^2\int_0^t\int_{\mathbb R} G_{t-r}(x-z)^2 \|Du_n(r,z)\|_{L^p(\Omega;\mathcal H)}^2 \,dz\,dr <\infty. 
    \end{aligned} \end{equation}
    Thus, 
    \[ \mathbf 1_{[0,t]}(r) G_{t-r}(x-z)\sigma(u_n(r,z))\in\mathbb D^{1,p}(\mathcal H). 
    \] The Malliavin differentiation rule for predictable stochastic integrals therefore implies that 
    \[ u_{n+1}(t,x)\in\mathbb D^{1,p}, 
    \] 
    and, for almost every \((s,y)\in[0,t)\times\mathbb R\), \(\mathbb P\)-a.s., 
    \[ \begin{aligned} D_{s,y}u_{n+1}(t,x) &= G_{t-s}(x-y)\sigma(u_n(s,y))\\ &\quad+ \int_s^t\int_{\mathbb R} G_{t-r}(x-z)\Sigma_n(r,z) D_{s,y}u_n(r,z)\,W(dr,dz). \end{aligned} 
    \] 
    For almost every \((s,y)\in(t,T]\times\mathbb R\), \(\mathbb P\)-a.s.,
    \[ D_{s,y}u_{n+1}(t,x)=0. 
    \]
    
It remains to obtain a uniform bound for the Malliavin derivatives. For $n\geq 0$, we denote by 
	\[
	V_n(t)
		:=
		\sup_{x\in\mathbb R}
		\left\|
		\left(
		\int_0^t\int_{\mathbb R}
		|D_{s,y}u_n(t,x)|^2\,dy\,ds
		\right)^{1/2}
		\right\|_p^2. 
	\]
    Using the representation of \(Du_{n+1}(t,x)\) and the triangle
inequality in \(L^p(\Omega;\mathcal H)\), we have
\[
\begin{aligned}
\|Du_{n+1}(t,x)\|_{L^p(\Omega;\mathcal H)}^2
&\le
C\left\|
\left(
\int_0^t\int_{\mathbb R}
G_{t-s}(x-y)^2
|\sigma(u_n(s,y))|^2\,dy\,ds
\right)^{1/2}
\right\|_p^2\\
&+
C\left\|
\int_0^t\int_{\mathbb R}
D\left(\mathbf 1_{[0,t]}(r)
G_{t-r}(x-z)\sigma(u_n(r,z))\right)\,W(dr,dz)
\right\|_{L^p(\Omega;\mathcal H)}^2.
\end{aligned}
\]
 The first term is bounded
by \(C_{p,T}\) by \eqref{eq:integrand-n-Lp}. For the stochastic term,
the Hilbert-valued Burkholder--Davis--Gundy inequality, followed by
Minkowski's inequality, gives
\[
\begin{aligned}
&\left\|
\int_0^t\int_{\mathbb R}
D\left(\mathbf 1_{[0,t]}(r)
G_{t-r}(x-z)\sigma(u_n(r,z))\right)\,W(dr,dz)
\right\|_{L^p(\Omega;\mathcal H)}^2\\
&\qquad\le
C_p
\left\|D\left(\mathbf 1_{[0,t]}(r)
G_{t-r}(x-z)\sigma(u_n(r,z))\right)\right\|_
{L^p(\Omega;\mathcal H\otimes\mathcal H)}^2\\
&\qquad\le
C\int_0^t\int_{\mathbb R}
G_{t-r}(x-z)^2V_n(r)\,dz\,dr.
\end{aligned}
\]

	By the heat-kernel estimate in Lemma \ref{le:G^2}, we have
	\[
	\int_{\mathbb R}G_{t-r}^2(z)\,dz
	\leq
	C(t-r)^{-\frac \beta \alpha},
	\]
	and since $\beta /\alpha<1$, the kernel
	$(t-r)^{-\beta /\alpha}$ is integrable on $[0,t]$. 
    Consequently,    \begin{equation}\label{eq:Vn-Volterra} 
    V_{n+1}(t) \le C_{p,T} + C\int_0^t (t-r)^{-\beta/\alpha}V_n(r)\,dr. \end{equation}
    Since \(V_n\) is bounded on \([0,T]\) by the induction hypothesis and
\(\beta/\alpha<1\), \eqref{eq:Vn-Volterra} shows that \(V_{n+1}(t)<\infty\).
This completes the induction. We next derive a bound independent of \(n\).
    For \(N\ge1\), define \[ W_N(t):=\max_{0\le n\le N}V_n(t). \] Since \(V_0(t)=0\),  \eqref{eq:Vn-Volterra} implies that 
    \[ W_N(t) \le C_{p,T} + C\int_0^t(t-r)^{-\frac{\beta}{\alpha}}W_N(r)\,dr. 
    \] The singular Gronwall inequality therefore gives 
    \[ W_N(t) \le C_{p,T} E_{1-\beta/\alpha}\left(C\,\Gamma\left(1-\frac{\beta}{\alpha}\right)t^{1-\frac{\beta}{\alpha}}\right), \qquad t\in[0,T], 
    \] 
    where \(E_{1-\beta/\alpha}\) denotes the Mittag--Leffler function. Since the right-hand side is independent of \(N\), we obtain \begin{equation}\label{eq:uniform-Picard-D} \sup_{n\ge0}\sup_{t\in[0,T]}V_n(t)  <\infty. 
    \end{equation}
    For every fixed \((t,x)\), the sequence
		\(\{Du_n(t,x)\}_{n\ge0}\) is therefore bounded in
		\(L^p(\Omega;\mathcal H)\). 
    Since \(p>1\), this space is reflexive. Hence, there exist a subsequence, still denoted by \(n\), and an element \(U(t,x)\in L^p(\Omega;\mathcal H)\) such that 
        \[ Du_n(t,x) \rightharpoonup U(t,x) \quad\text{weakly in }L^p(\Omega;\mathcal H). 
        \] 
        Since \(u_n(t,x)\to u(t,x)\) in \(L^p(\Omega)\), the weak closedness
of the Malliavin derivative yields
        \[ u(t,x)\in\mathbb D^{1,p}, \qquad Du(t,x)=U(t,x). 
        \] 
        
        Furthermore, by weak lower semicontinuity and \eqref{eq:uniform-Picard-D}, \begin{equation}\label{eq:uniform-limit-D} \sup_{(t,x)\in[0,T]\times\mathbb R} \|Du(t,x)\|_{L^p(\Omega;\mathcal H)} <\infty. 
        \end{equation} 
        By the Lipschitz chain rule, there exists a bounded predictable random field \(\Sigma\), with
\(|\Sigma|\le L\),
        such that 
        \[ D\sigma(u(r,z)) = \Sigma(r,z)Du(r,z). 
        \]

        By the uniform moment bound for \(u\), \eqref{eq:uniform-limit-D}, and the same estimates as in \eqref{eq:integrand-n-Lp} and \eqref{eq:D-integrand-n-Lp}, with \(u\) in place of \(u_n\), we obtain 
        \[ \mathbf 1_{[0,t]}(r) G_{t-r}(x-z)\sigma(u(r,z))\in\mathbb D^{1,p}(\mathcal H). 
        \]
        Therefore, the Malliavin differentiation rule for predictable stochastic integrals applies to \eqref{eq:solution}. Consequently, \eqref{eq:Malliavin-derivative} follows, and the proof is complete.
\end{proof}
We next give a technical estimate that will play an essential role in the sequel. Its proof is postponed to Appendix~\ref{6}.
\begin{lemma}\label{le:malliavin}
	Let \(u\) be the solution to equation \eqref{eq:main equation}, and let \(G\) denote the
	corresponding heat kernel defined in \eqref{eq:G=integral}. Assume that \(d=1\).
	Fix \(T>0\), \(p\ge2\), and let \(0\le s<r\le T\), \(y,z\in\mathbb R\).
	Assume that the Malliavin derivative of \(u\) satisfies
	\begin{equation}\label{eq:malliavin-iterate-G}
		\|D_{s,y}u(r,z)\|_p^2
		\le C G_{r-s}^2(z-y)
		+ C\int_s^r\int_{\mathbb R}
		G_{r-r_1}^2(z-z_1)\,
		\|D_{s,y}u(r_1,z_1)\|_p^2\,dz_1\,dr_1,
	\end{equation}
	for some constant \(C>0\) depending on \(T\) and \(p\). Then there exists a
	constant \(C_{T,p}>0\), depending on \(T\) and \(p\), such that
	\begin{equation}\label{eq:bound-D-square}
		\|D_{s,y}u(r,z)\|_p^2
		\le C_{T,p}\,(r-s)^{-\frac \beta \alpha}\,G_{r-s}(z-y),
	\end{equation}
	and consequently
	\begin{equation}\label{eq:bound-D-half}
		\|D_{s,y}u(r,z)\|_p
		\le C_{T,p}\,(r-s)^{-\frac{\beta}{2\alpha}}\,G_{r-s}(z-y)^{\frac 1 2}.
	\end{equation}
\end{lemma}

Additionally, the two-parameter Clark–Ocone formula will play a fundamental role in the proof of our results; see, for instance, Chen \emph{et al.} \cite[Proposition 6.3]{chen2021spatial}. Let $\{\mathcal{ F}_s\}_{s\geq0}$ denote the filtration generated by the space-time white noise $W$. For every
\(F\in\mathbb D^{1,2}\),
\[
F
=
\mathbb E[F]
+
\int_0^\infty\int_{\mathbb R}
\mathbb E[D_{s,z}F\mid\mathcal F_s]\,W(ds,dz),
\qquad \text{a.s.}
\]
As a direct consequence of the Clark--Ocone formula and the conditional Jensen inequality, we obtain the following Poincar\'e-type covariance bound:
\begin{equation}\label{eq:Poincare-type-ineq}
	|\operatorname{Cov}(F,G)|
	\le
	\int_0^\infty\int_{\mathbb R}
	\|D_{s,z}F\|_2\,\|D_{s,z}G\|_2\,dz\,ds \quad \text{for all}\quad F,G\in\mathbb D^{1,2}.
\end{equation}

\subsection{Stein's method}
Stein's method provides a powerful probabilistic framework for normal approximation by quantifying the distance between the distribution of a random variable and the standard normal law under various probability metrics. Its combination with Malliavin calculus, known as the Malliavin--Stein method, yields explicit bounds for these distances in terms of Malliavin derivatives and divergence operators. In this paper, we use this approach to obtain a quantitative bound in total variation distance. We recall that the total variation distance between the laws of two real-valued random variables $ F $ and $ G $, is defined by
\[d_{TV}(F,G):=\sup_{B\in \mathcal{B}(\mathbb{R})}|P(F\in B)-P(G\in B)|\]
where $ \mathcal{B}(\mathbb{R}) $ denotes the collection of all Borel sets in $ \mathbb{R} $. 

The following result is a version of Theorem 3.3.1 of Nourdin and Peccati~\cite{nourdin2012normal}.
\begin{theorem}
	For $ Z\sim N(0,1) $ and for any integrable random variable $ F $,
	\[d_{TV}(F,Z)\leq \sup_{f\in\mathcal{F}_{TV}}|\mathbb{E}[f^\prime(F)]-\mathbb{E}[Ff(F)]|,\]
	where $ \mathcal{F}_{TV}=\{f:||f||_\infty\leq \sqrt{\pi/2}, ||f^\prime||_\infty\leq2\}. $  
\end{theorem}
The combination of Stein's method for normal approximations with Malliavin calculus leads to the following bound on the total variation distance whose proof can be found in Proposition 2.2 of Huang \emph{et al.}\cite{huang2020}.
\begin{proposition}\label{pro:2.1}
	Let $ F=\delta(v) $ for some $ \mathcal{H} $-valued random variable $ v $ which belongs to Dom$\delta$. Assume $ \mathbb{E}[F^2]=1 $ and $ F\in \mathbb{D}^{1,2} $. Let $ Z\sim N(0,1) $. Then we have
	\begin{equation}
		d_{TV}(F,Z)\leq 2\sqrt{\mathrm{Var}\langle DF,v\rangle_\mathcal{H}}.
	\end{equation}
\end{proposition}

In order to prove Theorem \ref{FCLT}, we need the following proposition, which is a generalization of Theorem 6.1.2 in Nourdin and Peccati~\cite{nourdin2012normal}.
\begin{proposition}\label{pro:con-for-finite-distri}
	Let $F=(F^{(1)},\ldots,F^{(m)})$ be a random vector such that
	$F^{(i)}=\delta(v^{(i)})$ for $v^{(i)}\in \operatorname{Dom}\delta$,
	$i=1,\ldots,m$. Assume $F^{(i)}\in \mathbb{D}^{1,2}$ for
	$i=1,\ldots,m$. Let $Z$ be an $m$-dimensional centered Gaussian vector
	with covariance matrix $(C_{i,j})_{1\leq i,j\leq m}$. For any $C^2$
	function $h:\mathbb{R}^m\to\mathbb{R}$ with bounded second partial
	derivatives, we have
	\[
	\left|
	\mathbb{E}h(F)-\mathbb{E}h(Z)
	\right|
	\leq
	\frac{1}{2}\|h''\|_{\infty}
	\sqrt{
		\sum_{i,j=1}^{m}
		\mathbb{E}\left[
		\left(
		C_{i,j}
		-
		\left\langle DF^{(i)},v^{(j)}\right\rangle_{\mathcal{H}}
		\right)^2
		\right]
	},
	\]
	where
	\[
	\|h''\|_{\infty}
	=
	\max_{1\leq i,j\leq m}
	\sup_{x\in\mathbb{R}^m}
	\left|
	\frac{\partial^2 h}{\partial x_i\partial x_j}(x)
	\right|.
	\]
\end{proposition}
\begin{proof}
	See Proposition 2.3 in Huang \emph{et al.} \cite{huang2020}.
\end{proof}

\section{Covariance asymptotics for spatial averages}\label{3}
In this section, we identify the asymptotic covariance structure of the
centered spatial averages. This result determines both the normalization in
the quantitative central limit theorem and the covariance function of the
limiting process in the functional central limit theorem.

\begin{proposition}\label{pro:cov_GR}
		Consider the random field $\{u(t,x), t\geq0, x\in\R\}$ given by \eqref{eq:solution}. Recall
		\[
		\xi(r):=\mathbb E[\sigma(u(r,x))^2],
		\quad\text{and denote}\quad
		G_R(t):=\int_{-R}^R u(t,x)\,dx-2R.
		\]
		Then, for every \(s,t\ge0\),
		\begin{equation}
			\lim_{R\to\infty}\frac{\Cov(G_R(t),G_R(s))}{R}
			=
			2\int_0^{s\wedge t}\xi(r)\,dr.
		\end{equation}
	\end{proposition}

	\begin{proof}
		By the mild solution \eqref{eq:solution} and stochastic Fubini theorem,
		\[
		G_R(t)
		=
		\int_0^t\int_{\mathbb R}
		\varphi_{t-r,R}(y)\sigma(u(r,y))\,W(dr,dy),
		\]
		where
		\[
		\varphi_{s,R}(y):=\int_{-R}^R G_s(x-y)\,dx,\quad y\in\R .
		\]
		Since \(\mathbb E[u(t,x)]=1\), both \(G_R(t)\) and \(G_R(s)\) are centered.
		Hence, by It\^o's isometry,
		\begin{align*}
			\Cov(G_R(t),G_R(s))
			&=
			\int_0^{s\wedge t}\int_{\mathbb R}
			\varphi_{t-r,R}(y)\varphi_{s-r,R}(y)
			\mathbb E[\sigma(u(r,y))^2]
			\,dy\,dr \\
			&=
			\int_0^{s\wedge t}\xi(r)
			\int_{\mathbb R}
			\varphi_{t-r,R}(y)\varphi_{s-r,R}(y)
			\,dy\,dr,
		\end{align*}
		where we used the fact that \(\xi(r)\) does not depend on \(y\).
		
		Now, using the representation \eqref{eq:G=integral}
		\[
		G_\tau(x)=\int_0^\infty p_\alpha(\ell,x)f_{E_\tau}(\ell)\,d\ell,
		\qquad \tau>0,
		\]
		we obtain
		\begin{align*}
			\Cov(G_R(t),G_R(s))
			&=
			\int_{[-R,R]^2}\int_0^{s\wedge t}\xi(r)
			\int_{\mathbb R}\int_0^\infty\int_0^\infty
			p_\alpha(\ell_1,x-y)p_\alpha(\ell_2,x'-y) \\
			&\hspace{2.7cm}\times
			f_{E_{t-r}}(\ell_1)f_{E_{s-r}}(\ell_2)
			\,d\ell_1d\ell_2\,dy\,dr\,dx\,dx'.
		\end{align*}
		According to \eqref{eq:f-G-integral}, that is
		\[
		f_{E_\tau}(\ell)
		=
		\tau\beta^{-1}\ell^{-1-1/\beta}
		g_\beta\!\left(\tau \ell^{-1/\beta}\right),
		\]
		and making the changes of variables
		\[
		\ell_1=\left(\frac{t-r}{z_1}\right)^\beta,
		\qquad
		\ell_2=\left(\frac{s-r}{z_2}\right)^\beta,
		\]
		we get
		\begin{align*}
			\Cov(G_R(t),G_R(s))
			&=
			\int_{[-R,R]^2}\int_0^{s\wedge t}\xi(r)
			\int_{[0,\infty)^2}\int_{\mathbb R}
			p_\alpha\!\left(\left(\frac{t-r}{z_1}\right)^\beta,x-y\right)
			p_\alpha\!\left(\left(\frac{s-r}{z_2}\right)^\beta,x'-y\right) \\
			&\hspace{2.8cm}\times g_\beta(z_1)g_\beta(z_2)
			\,dy\,dz_1dz_2\,dr\,dx\,dx'.
		\end{align*}
		By the semigroup property, we further obtain
		\begin{align*}
			\Cov(G_R(t),G_R(s))
			&=
			\int_{[-R,R]^2}\int_0^{s\wedge t}\xi(r)
			\int_{[0,\infty)^2}
			p_\alpha\!\left(
			\left(\frac{t-r}{z_1}\right)^\beta
			+
			\left(\frac{s-r}{z_2}\right)^\beta,
			x-x'
			\right) \\
			&\hspace{2.8cm}\times g_\beta(z_1)g_\beta(z_2)
			\,dz_1dz_2\,dr\,dx\,dx'.
		\end{align*}

		 Hence
		\begin{align*}
			\frac{\Cov(G_R(t),G_R(s))}{R}
			&=
			2\int_0^{s\wedge t}\xi(r)\int_0^\infty\int_0^\infty\int_0^{2R}
			p_\alpha\!\left(
			\left(\frac{t-r}{z_1}\right)^\beta
			+
			\left(\frac{s-r}{z_2}\right)^\beta,
			x
			\right) \\
			&\hspace{3.2cm}\times
			\left(2-\frac{x}{R}\right)
			g_\beta(z_1)g_\beta(z_2)
			\,dx\,dz_1dz_2\,dr.
		\end{align*}

Since
\[
\left(2-\frac{x}{R}\right)\mathbf 1_{[0,2R]}(x)\to2,
\qquad
0\leq
\left(2-\frac{x}{R}\right)\mathbf 1_{[0,2R]}(x)\leq2,
\]
and since \(\xi\) is bounded, by the dominated convergence theorem, the symmetry of
\(p_\alpha\), and the fact that \(p_\alpha\) and \(g_\beta\) are
probability density functions, we obtain
\[
\lim_{R\to\infty}
\frac{\Cov(G_R(t),G_R(s))}{R}
=
2\int_0^{s\wedge t}\xi(r)\,dr.
\] 
		This completes the proof.
\end{proof}
As an immediate consequence of Proposition \ref{pro:cov_GR}, we have the following corollary and the proof is omitted.
\begin{corollary}\label{cor:var_GR}
	For every \(t\ge0\),
	\[
	\lim_{R\to\infty}
	\frac{\operatorname{Var}(G_R(t))}{R}
	=
	2\int_0^t\xi(r)\,dr.
	\]
	In particular, if
	\[
	\int_0^t\xi(r)\,dr>0,
	\]
	then
	\[
	\sigma_R^2=\operatorname{Var}(G_R(t))
	\sim
	2R\int_0^t\xi(r)\,dr,
	\qquad R\to\infty.
	\]
\end{corollary}
\begin{remark}\label{rem:nondegenerate-variance}
	The assumption \(\sigma(1)\neq0\) excludes the degenerate case
	\(u(t,x)\equiv1\). Moreover, by the \(L^2\)-continuity of the mild
	solution at the initial time, \(u(r,x)\to1\) in \(L^2(\Omega)\) as
	\(r\downarrow0\). Hence, by the Lipschitz continuity of \(\sigma\),
	\[
		\sigma(u(r,x))\to\sigma(1)
		\quad\text{in }L^2(\Omega),
	\]
	and consequently
	\[
		\xi(r)=\mathbb E[\sigma(u(r,x))^2]\to \sigma(1)^2>0.
	\]
	Thus \(\xi(r)>0\) for all sufficiently small \(r>0\), and hence
	\[
		\int_0^t\xi(r)\,dr>0,
		\qquad t>0.
	\]
    Therefore, the limiting variance in Corollary~\ref{cor:var_GR}
is non-degenerate for every \(t>0\).

Moreover, for every \(t>0\) and \(R>0\), It\^o's isometry gives
\[
\sigma_R^2
=
\int_0^t
\xi(r)
\left\|
\varphi_{t-r,R}
\right\|_{L^2(\mathbb R)}^2
\,dr.
\]
Since \(\xi(r)>0\) for all sufficiently small \(r>0\), while
\(\varphi_{\tau,R}\geq0\) and
\[
\int_{\mathbb R}\varphi_{\tau,R}(y)\,dy=2R>0,
\]
we have
\[
\left\|\varphi_{\tau,R}\right\|_{L^2(\mathbb R)}^2>0
\]
for every \(\tau>0\). Consequently,
\[
\sigma_R^2>0
\qquad\text{for every }t>0\text{ and }R>0.
\]
\end{remark}

\section{Proof of Theorem \ref{QCLT}}\label{4}
For \(\tau>0\), recall that
\[
\varphi_{\tau,R}(y):=\int_{-R}^R G_\tau(x-y)\,dx,
\qquad y\in\mathbb R.
\]
Since \(G_\tau\ge0\) and \(\int_{\mathbb R}G_\tau(x)\,dx=1\), we have
\[
0\le \varphi_{\tau,R}(y)\le 1,
\qquad y\in\mathbb R,
\]
and
\[
\int_{\mathbb R}\varphi_{\tau,R}(y)\,dy
=
\int_{\mathbb R}\int_{-R}^R G_\tau(x-y)\,dx\,dy
=
\int_{-R}^R\int_{\mathbb R}G_\tau(x-y)\,dy\,dx
=
2R.
\]
Hence
\begin{equation}\label{eq:norm-phi}
	\|\varphi_{\tau,R}\|_{L^2(\mathbb R)}^2
	=
	\int_{\mathbb R}\varphi_{\tau,R}(y)^2\,dy
	\le
	\int_{\mathbb R}\varphi_{\tau,R}(y)\,dy
	=
	2R.
\end{equation}
To apply Proposition \ref{pro:2.1}, it suffices to represent \(F_R(t)\) as a
Skorohod integral \(\delta(v_R)\). Recall that in our case, applying the stochastic Fubini theorem,
\begin{align*}
	F_R(t)&=\frac{1}{\sigma_R}\left(\int_{-R}^R u(t,x)dx-2R\right)\\
	&=\frac{1}{\sigma_R}\int_{-R}^R\int_0^t\int_{\mathbb{R}}
	G_{t-s}(x-y)\sigma(u(s,y))W(ds,dy)dx\\
	&=\int_0^t\int_{\mathbb{R}}
	\frac{\varphi_{t-s,R}(y)}{\sigma_R}\sigma(u(s,y))W(ds,dy).
\end{align*}
Then, in view of the definition of $ \delta(X) $, we have, for any fixed $ t> 0 $, $ F_R(t)=\delta(v_R) $, where
\[v_R(s,y)
=\mathbbm{1}_{[0,t]}(s)\frac{\varphi_{t-s,R}(y)}{\sigma_R} \sigma(u(s,y)).\]
In fact, the process \(v_R\) is predictable. By the uniform moment
	bound for \(u\) and \eqref{eq:norm-phi},
	we have
	\[
	\begin{aligned}
		\mathbb E\|v_R\|_{\mathcal H}^2
		&=
		\frac{1}{\sigma_R^2}
		\int_0^t\int_{\mathbb R}
		\varphi_{t-s,R}(y)^2
		\mathbb E\bigl[|\sigma(u(s,y))|^2\bigr]
		\,dy\,ds\\
		&\leq
		\frac{C}{\sigma_R^2}
		\int_0^t
		\|\varphi_{t-s,R}\|_{L^2(\mathbb R)}^2\,ds\\
		&\leq
		\frac{2CRt}{\sigma_R^2}
		<\infty.
	\end{aligned}
	\]
	Hence \(v_R\in\operatorname{Dom}\delta\). Then
	\(\delta(v_R)\) coincides with the corresponding Walsh integral, and
	therefore \(F_R(t)=\delta(v_R)\) is valid.
Moreover,
\[D_{s,y}F_R(t)=\mathbbm{1}_{[0,t]}(s)\frac{1}{\sigma_R}\int_{-R}^RD_{s,y}u(t,x)dx.\]
By (\ref{eq:Malliavin-derivative}), we know that
\begin{align*}
	\langle DF_R(t),  v_R\rangle_\mathcal{H}
	&=\frac{1}{\sigma_R^2}\int_0^t\int_{\mathbb{R}}\varphi_{t-s,R}(y)\sigma(u(s,y))\int_{-R}^RD_{s,y}u(t,x)\,dx\,dy\,ds\\
	&=\frac{1}{\sigma_R^2}\int_0^t\int_{\mathbb{R}}\varphi_{t-s,R}(y)\sigma(u(s,y))\left[\sigma(u(s,y))\varphi_{t-s,R}(y)\right.\\
	&\quad\quad+\left. \int_s^t\int_{\mathbb{R}}\varphi_{t-r,R}(z)\Sigma(r,z)D_{s,y}u(r,z)W(dr,dz)\right]\,dy\,ds\\
	&=\frac{1}{\sigma_R^2}\int_0^t\int_{\mathbb{R}}\varphi_{t-s,R}^2(y)\sigma^2(u(s,y))dyds\\
	&\quad+\frac{1}{\sigma_R^2}\int_0^t\int_{\mathbb{R}}\varphi_{t-s,R}(y)\sigma(u(s,y))\\
	&\quad\quad\times\left(\int_s^t\int_{\mathbb{R}}\varphi_{t-r,R}(z)\Sigma(r,z)D_{s,y}u(r,z)W(dr,dz)\right)\,dy\,ds.
\end{align*}
By using the fact that
\begin{equation}\label{eq:sqrt-Var-sqrt}
	\sqrt{\mathrm{Var}\left(\int_0^tX_sds\right)}\leq \int_0^t\sqrt{\mathrm{Var} X_s}ds,
\end{equation}
for any stochastic process $ \{X_s, s\in[0,t]\} $ such that $ \sqrt{\mathrm{Var} X_s} $ is integrable over $ [0,t] $, we have
\begin{align*}
	\sqrt{\mathrm{Var}(\langle DF_R(t), v_R\rangle)}&\leq \sqrt{2}(A_1+A_2),
\end{align*}
where
\begin{equation*}
	A_1= \frac{1}{\sigma_R^2}\int_0^t\left[\int_{\mathbb{R}^2}\varphi_{t-s,R}^2(y)\varphi_{t-s,R}^2(y')\text{Cov}(\sigma^2(u(s,y)),\sigma^2(u(s,y')))\,dy\,dy'\,\right]^{\frac{1}{2}}ds
\end{equation*}
and
\begin{align*}
	A_2&=\frac{1}{\sigma_R^2}\int_0^t\left[\int_{\mathbb{R}^2}\varphi_{t-s,R}(y)\varphi_{t-s,R}(y')\int_s^t\int_{\mathbb{R}}\varphi_{t-r,R}^2(z)\right.\\
	&\qquad\times\left.\mathbb{E}\left[\sigma(u(s,y))\sigma(u(s,y'))\Sigma^2(r,z)D_{s,y}u(r,z)D_{s,y'}u(r,z)\right]\,dz\,dr\,dy\,dy'\,\right]^{\frac{1}{2}}ds.
\end{align*}

Let us first estimate the term $A_2$. Let, for $p\geq2$,
\[K_p(t)=\sup_{0\leq s\leq t}\sup_{y\in\mathbb{R}}||\sigma(u(s,y))||_p=\sup_{0\leq s\leq t}||\sigma(u(s,0))||_p.\]
It follows immediately that
\begin{multline*}
	\left|\mathbb{E}\left[\sigma(u(s,y))\sigma(u(s,y'))\Sigma^2(r,z)D_{s,y}u(r,z)D_{s,y'}u(r,z)\right]\right|\\
	\leq CK_4^2(t)L^2||D_{s,y}u(r,z)||_4 ||D_{s,y'}u(r,z)||_4.
\end{multline*}

We next estimate the moments of the Malliavin derivative.  By using Burkholder's inequality and Minkowski's inequality, we have
\begin{align*}
	||D_{s,y}u(r,z)||_p&\leq G_{r-s}(z-y)K_p(t)+Lc_p\left(\mathbb{E}\left|\int_s^r\int_{\mathbb{R}}G_{r-r_1}^2(z-z_1)|D_{s,y}u(r_1,z_1)|^2dz_1dr_1\right|^{\frac{p}{2}}\right)^{\frac{1}{p}}\\
	&\leq G_{r-s}(z-y)K_p(t)+Lc_p\left[\int_s^r\int_{\mathbb{R}}G_{r-r_1}(z-z_1)^2||(D_{s,y}u(r_1,z_1))^2||_{\frac{p}{2}}dz_1dr_1\right]^{\frac{1}{2}}\\
	&\leq G_{r-s}(z-y)K_p(t)+Lc_p\left[\int_s^r\int_{\mathbb{R}}G_{r-r_1}(z-z_1)^2||D_{s,y}u(r_1,z_1)||_p^2 dz_1dr_1\right]^{\frac{1}{2}},
\end{align*}
which implies
\begin{equation*}
	||D_{s,y}u(r,z)||_p^2\leq 2G_{r-s}^2(z-y)K_p^2(t)\,+\,2L^2c_p^2\int_s^r\int_{\mathbb{R}}G_{r-r_1}^2(z-z_1)||D_{s,y}u(r_1,z_1)||_p^2 \,dz_1\,dr_1.
\end{equation*}

Thus, Lemma \ref{le:malliavin} yields
\begin{equation}\label{eq:bound of D_{s,y}}
	||D_{s,y}u(r,z)||_p\leq C_{T,p}\,(r-s)^{-\frac{\beta}{2\alpha}}\,G_{r-s}(z-y)^{\frac12}.
\end{equation}

Therefore, for all sufficiently large \(R\), Corollary \ref{cor:var_GR} and (\ref{eq:bound of D_{s,y}}) yield
\begin{align*}
	A_2&\leq \frac{C}{R}\int_0^t\Bigg[\int_{\mathbb{R}^2}\int_s^t\int_{\mathbb{R}}
	\varphi_{t-s,R}(y)\varphi_{t-s,R}(y')\varphi_{t-r,R}^2(z)\\
	&\qquad\times (r-s)^{-\frac{\beta}{\alpha}}
	\,G_{r-s}(z-y)^{\frac 1 2}G_{r-s}(z-y')^{\frac 1 2}
	\,dz\,dr\,dy\,dy'\Bigg]^{\frac{1}{2}}ds.
\end{align*}

Using Fubini's theorem, we obtain
\[
\begin{aligned}
	A_2
	&\le \frac{C}{R}\int_0^t
	\Bigg[
	\int_s^t (r-s)^{-\frac{\beta}{\alpha}}
	\int_{\mathbb R}
	\varphi_{t-r,R}(z)^2 \\
	&\qquad\times
	\Bigg(
	\int_{\mathbb R}\varphi_{t-s,R}(y)G_{r-s}(z-y)^{\frac 1 2}\,dy
	\Bigg)
	\Bigg(
	\int_{\mathbb R}\varphi_{t-s,R}(y')G_{r-s}(z-y')^{\frac 1 2}\,dy'
	\Bigg)
	\,dz\,dr
	\Bigg]^{\frac 1 2}ds \\
	&= \frac{C}{R}\int_0^t
	\Bigg[
	\int_s^t (r-s)^{-\frac{\beta}{\alpha}}
	\int_{\mathbb R}\varphi_{t-r,R}(z)^2
	\Bigg(
	\int_{\mathbb R}\varphi_{t-s,R}(y)G_{r-s}(z-y)^{\frac 1 2}\,dy
	\Bigg)^2
	\,dz\,dr
	\Bigg]^{\frac 1 2}ds.
\end{aligned}
\]

Using \(0\le \varphi_{t-s,R}\le 1\) together with \ref{it:G-property-3} in Lemma \ref{lem:basic-G-estimates}, we deduce that
for every \(z\in\mathbb R\),
\[
\begin{aligned}
	\int_{\mathbb R}\varphi_{t-s,R}(y)G_{r-s}(z-y)^{1/2}\,dy
	&\le
	\int_{\mathbb R}G_{r-s}(z-y)^{1/2}\,dy \\
	&=
	\int_{\mathbb R}G_{r-s}(y)^{1/2}\,dy
	\le
	C(r-s)^{\frac{\beta}{2\alpha}}.
\end{aligned}
\]
Therefore,
\[
\Bigg(
\int_{\mathbb R}\varphi_{t-s,R}(y)G_{r-s}(z-y)^{1/2}\,dy
\Bigg)^2
\le C(r-s)^{\frac \beta \alpha}.
\]

Substituting this into the previous estimate yields
\[
\begin{aligned}
	A_2
	&\le \frac{C}{R}\int_0^t
	\Bigg[
	\int_s^t (r-s)^{-\frac \beta \alpha}(r-s)^{\frac \beta \alpha}
	\int_{\mathbb R}\varphi_{t-r,R}(z)^2\,dz\,dr
	\Bigg]^{1/2}ds \\
	&\le \frac{C}{R}\int_0^t
	\Bigg[
	\int_s^t 	\|\varphi_{t-r,R}\|_{L^2(\mathbb R)}^2\,dr
	\Bigg]^{1/2}ds \\
	&\le \frac{C}{R}\int_0^t
	\left[
	R(t-s)
	\right]^{\frac 1 2}ds
	\,\leq\, \frac{C}{\sqrt{R}}.
\end{aligned}
\]

Consequently,
\[
A_2\to0
\qquad\text{as }R\to\infty.
\]

For \(A_1\), we first estimate
\[
\operatorname{Cov}\left(
\sigma^2(u(s,y)),\sigma^2(u(s,y'))
\right).
\]
By the Poincar\'e-type covariance inequality \eqref{eq:Poincare-type-ineq}
and the fact that \(D_{r,z}u(s,y)=0\) for \(r>s\), we have
\[
\begin{aligned}
	&\left|
	\operatorname{Cov}\left(
	\sigma^2(u(s,y)),\sigma^2(u(s,y'))
	\right)
	\right|\\
	&\quad\le
	\int_0^s\int_{\mathbb R}
	\left\|
	D_{r,z}\big(\sigma^2(u(s,y))\big)
	\right\|_2
	\left\|
	D_{r,z}\big(\sigma^2(u(s,y'))\big)
	\right\|_2
	\,dz\,dr.
\end{aligned}
\]
By the product rule and  chain rule for Lipschitz functions (see Proposition 1.2.4 in Nualart~\cite{Nualart2006}), there exists a random
field \(\Sigma\), satisfying \(|\Sigma(s,y)|\le L\), such that
\[
D_{r,z}\big(\sigma^2(u(s,y))\big)
=
2\sigma(u(s,y))\Sigma(s,y)D_{r,z}u(s,y).
\]
Consequently, by H\"older's inequality,
\[
\begin{aligned}
	\left\|
	D_{r,z}\big(\sigma^2(u(s,y))\big)
	\right\|_2
	&\le
	2L\|\sigma(u(s,y))\|_4
	\|D_{r,z}u(s,y)\|_4\\
	&\le
	2LK_4(t)\|D_{r,z}u(s,y)\|_4.
\end{aligned}
\]
Therefore,
\[
\begin{aligned}
	&\left|
	\operatorname{Cov}\left(
	\sigma^2(u(s,y)),\sigma^2(u(s,y'))
	\right)
	\right|\\
	&\quad\le
	4L^2K_4^2(t)
	\int_0^s\int_{\mathbb R}
	\|D_{r,z}u(s,y)\|_4
	\|D_{r,z}u(s,y')\|_4
	\,dz\,dr.
\end{aligned}
\]
Applying Lemma~\ref{le:malliavin}, we obtain
\[
\begin{aligned}
	&\left|
	\operatorname{Cov}\left(
	\sigma^2(u(s,y)),\sigma^2(u(s,y'))
	\right)
	\right|\\
	&\quad\le
	C\int_0^s\int_{\mathbb R}
	(s-r)^{-\beta/\alpha}
	G_{s-r}(y-z)^{1/2}
	G_{s-r}(y'-z)^{1/2}
	\,dz\,dr.
\end{aligned}
\]

Then we get
\[
\begin{aligned}
	A_1
	&\le \frac{C}{R}\int_0^t
	\Bigg[
	\int_0^s (s-r)^{-\frac \beta\alpha}
	\int_{\mathbb R^2}
	\varphi_{t-s,R}(y)^2\varphi_{t-s,R}(y')^2 \\
	&\hspace{2.8cm}\times
	\Bigg(
	\int_{\mathbb R}G_{s-r}(z-y)^{1/2}G_{s-r}(z-y')^{1/2}\,dz
	\Bigg)
	\,dy\,dy'\,dr
	\Bigg]^{1/2}ds .
\end{aligned}
\]

Since \(0\le \varphi_{t-s,R}\le 1\), we have
	\[
	\int_{\mathbb R}\varphi_{t-s,R}(y)^4\,dy
	\le
	\int_{\mathbb R}\varphi_{t-s,R}(y)\,dy
	=2R.
	\]
	Therefore, by using Cauchy--Schwarz inequality and Young's inequality, we get
	\[
	\begin{aligned}
		&\int_{\mathbb R^2}
		\varphi_{t-s,R}(y)^2\varphi_{t-s,R}(y')^2
		\left(
		\int_{\mathbb R}
		G_{s-r}(z-y)^{1/2}G_{s-r}(z-y')^{1/2}\,dz
		\right)
		dy\,dy' \\
		&\quad =
		\int_{\mathbb R}
		\varphi_{t-s,R}(y)^2
		\left[
		\big(G_{s-r}^{1/2}*G_{s-r}^{1/2}\big)
		*
		\varphi_{t-s,R}^2
		\right](y)\,dy \\
		&\quad \le
		\|\varphi_{t-s,R}^2\|_{L^2(\mathbb R)}
		\left\|
		\big(G_{s-r}^{1/2}*G_{s-r}^{1/2}\big)
		*
		\varphi_{t-s,R}^2
		\right\|_{L^2(\mathbb R)} \\
		&\quad \le
		\|\varphi_{t-s,R}^2\|_{L^2(\mathbb R)}^2
		\left\|G_{s-r}^{1/2}*G_{s-r}^{1/2}\right\|_{L^1(\mathbb R)} \\
		&\quad \le
		C R (s-r)^{\beta/\alpha}.
	\end{aligned}
	\]
	
	Substituting this estimate into the bound for \(A_1\), we obtain
	\[
	\begin{aligned}
		A_1
		&\le
		\frac{C}{R}\int_0^t
		\left[
		R\int_0^s
		(s-r)^{-\beta/\alpha}(s-r)^{\beta/\alpha}\,dr
		\right]^{1/2}ds \\
		&=
		\frac{C}{R}\int_0^t
		(Rs)^{1/2}\,ds
		\le
		\frac{C}{\sqrt R}.
	\end{aligned}
	\]
	Consequently,
	\[
	A_1\to0
	\qquad\text{as }R\to\infty.
	\]
	
	Combining the estimates for \(A_1\) and \(A_2\) with the Malliavin--Stein bound, we conclude that
	\[
	d_{\mathrm{TV}}(F_R(t),Z)
	\le
	2\sqrt{\operatorname{Var}\left(\langle DF_R(t),v_R\rangle_{\mathcal H}\right)}
	\le
	\frac{C}{\sqrt R}.
	\]
	Choose \(R_0>0\) such that the preceding estimates hold for
	\(R\geq R_0\). For \(0<R<R_0\), the conclusion follows by enlarging
	the constant and using \(d_{\mathrm{TV}}(F_R(t),Z)\leq1\). This completes the proof.                                                             

\section{Proof of Theorem \ref{FCLT}}\label{5}
To prove Theorem~\ref{FCLT}, it suffices to establish the convergence of
finite-dimensional distributions and tightness. The tightness will follow from
Proposition~\ref{pro:time-increment-spatial-average} as well as the
Kolmogorov-Chentsov criterion.

\textbf{Step 1: Tightness.}
\begin{proposition}
	\label{pro:time-increment-spatial-average}
	Assume that $u_0\equiv 1$ and that $\sigma$ is Lipschitz continuous.
	Then, for every $p\geq 1$ and every $T>0$, there exists a constant
	$C=C(p,T)>0$ such that, for all $R>0$, and all $0\leq s<t\leq T$,
	\[
	\mathbb E\left[
	\left|G_R(t)-G_R(s)\right|^p
	\right]
	\leq
	C R^{p/2}(t-s)^{p/2}.
	\]
	Equivalently,
	\[
	\mathbb E\left[
	\left|
	\frac{G_R(t)-G_R(s)}{\sqrt R}
	\right|^p
	\right]
	\leq
	C(t-s)^{p/2}.
	\]
\end{proposition}
\begin{proof}
	Since $u_0\equiv 1$, the mild solution gives
	\[
	u(t,x)-1
	=
	\int_0^t\int_{\mathbb R}
	G_{t-r}(x-y)\sigma(u(r,y))\,W(dr,dy).
	\]
	Recall that
	\[
	\varphi_{\tau,R}(y):=\int_{-R}^{R}G_\tau(x-y)\,dx .
	\]
	Then, by stochastic Fubini's theorem,
	\[
	G_R(t)
	=
	\int_0^t\int_{\mathbb R}
	\varphi_{t-r,R}(y)\sigma(u(r,y))\,W(dr,dy).
	\]
	Hence, for $0\leq s<t\leq T$,
	\[
	G_R(t)-G_R(s)=B_{R,1}(s,t)+B_{R},2(s,t),
	\]
	where
	\[
	B_{R,1}(s,t)
	=
	\int_0^s\int_{\mathbb R}
	\left[
	\varphi_{t-r,R}(y)-\varphi_{s-r,R}(y)
	\right]
	\sigma(u(r,y))\,W(dr,dy),
	\]
	and
	\[
	B_{R,2}(s,t)
	=
	\int_s^t\int_{\mathbb R}
	\varphi_{t-r,R}(y)\sigma(u(r,y))\,W(dr,dy).
	\]
	
	We first estimate $B_{R},2(s,t)$. By the Burkholder--Davis--Gundy
	inequality and the uniform moment bounds for $u$, we have, for $p\geq2$,
	\[
	\|B_{R,2}(s,t)\|_{p}^{2}
	\leq
	C
	\int_s^t\int_{\mathbb R}
	|\varphi_{t-r,R}(y)|^2\,dy\,dr .
	\]
	Since
	\[
	\varphi_{t,R}(\cdot)=\mathbf 1_{[-R,R]}*G_t,
	\]
	Young's inequality and $\|G_t\|_{L^1(\mathbb R)}=1$ yield
	\[
	\|\varphi_{t,R}(\cdot)\|_{L^2(\mathbb R)}
	\leq
	\|\mathbf 1_{[-R,R]}\|_{L^2(\mathbb R)}
	\|G_t\|_{L^1(\mathbb R)}
	\leq C R^{1/2}.
	\]
	Therefore,
	\[
	\|B_{R,2}(s,t)\|_{p}^{2}
	\leq
	C R(t-s).
	\]
	That is,
	\[
	\|B_{R,2}(s,t)\|_{p}
	\leq
	C R^{1/2}(t-s)^{1/2}.
	\]
	
	It remains to estimate $B_{R,1}(s,t)$. By the BDG inequality again,
	\[
	\|B_{R,1}(s,t)\|_{p}^{2}
	\leq
	C
	\int_0^s\int_{\mathbb R}
	\left|
	\varphi_{t-r,R}(y)-\varphi_{s-r,R}(y)
	\right|^2dy\,dr .
	\]
	We claim that
	\[
	\int_0^s
	\left\|
	\varphi_{t-r,R}(\cdot)-\varphi_{s-r,R}(\cdot)
	\right\|_{L^2(\mathbb R)}^2\,dr
	\leq
	C R(t-s).
	\]
	Indeed, by the Parseval-Plancherel identity,
	\[
	\begin{aligned}
		&\int_0^s
		\left\|
		\varphi_{t-r,R}(\cdot)-\varphi_{s-r,R}(\cdot)
		\right\|_{L^2(\mathbb R)}^2\,dr \\
		&\quad
		=
		C\int_{\mathbb R}
		\left|
		\widehat{\mathbf 1}_{[-R,R]}(\xi)
		\right|^2
		\int_0^s
		\left|
		\widehat G_{t-r}(\xi)-\widehat G_{s-r}(\xi)
		\right|^2
		dr\,d\xi .
	\end{aligned}
	\]

By \eqref{eq:G-Fourier} and making the change of variables \(\tau=s-r\), we have
		\[
		\begin{aligned}
			&\int_0^s
			\left|
			\widehat G_{t-r}(\xi)-\widehat G_{s-r}(\xi)
			\right|^2\,dr  \\
			&\quad =
			\int_0^s
			\left|
			E_\beta\left(-\nu|\xi|^\alpha(\tau+t-s)^\beta\right)
			-
			E_\beta\left(-\nu|\xi|^\alpha\tau^\beta\right)
			\right|^2\,d\tau .
		\end{aligned}
		\]
		Since \(E_\beta(-x)\) is completely monotone on \([0,\infty)\), it is
		decreasing. Moreover, by the uniform estimate of the Mittag--Leffler
		function, \(0\le E_\beta(-x)\le 1\) for \(x\ge0\). Therefore, for every
		\(\tau\ge0\),
		\[
		0\le
		E_\beta\left(-\nu|\xi|^\alpha\tau^\beta\right)
		-
		E_\beta\left(-\nu|\xi|^\alpha(\tau+t-s)^\beta\right)
		\le 1.
		\]
		It follows that
		\[
		\begin{aligned}
			&\int_0^s
			\left|
			E_\beta\left(-\nu|\xi|^\alpha(\tau+t-s)^\beta\right)
			-
			E_\beta\left(-\nu|\xi|^\alpha\tau^\beta\right)
			\right|^2\,d\tau \\
			&\quad \le
			\int_0^s
			\left[
			E_\beta\left(-\nu|\xi|^\alpha\tau^\beta\right)
			-
			E_\beta\left(-\nu|\xi|^\alpha(\tau+t-s)^\beta\right)
			\right]\,d\tau \\
			&\quad =
			\int_0^s
			E_\beta\left(-\nu|\xi|^\alpha\tau^\beta\right)\,d\tau
			-
			\int_{t-s}^{t}
			E_\beta\left(-\nu|\xi|^\alpha \tau^\beta\right)\,d\tau \\
			&\quad =
			\int_0^{t-s}
			E_\beta\left(-\nu|\xi|^\alpha \tau^\beta\right)\,d\tau
			-
			\int_s^t
			E_\beta\left(-\nu|\xi|^\alpha \tau^\beta\right)\,d\tau \\
			&\quad \le
			\int_0^{t-s}
			E_\beta\left(-\nu|\xi|^\alpha \tau^\beta\right)\,d\tau
			\le t-s.
		\end{aligned}
		\]
		Consequently,
		\[
		\int_0^s
		\left|
		\widehat G_{t-r}(\xi)-\widehat G_{s-r}(\xi)
		\right|^2\,dr
		\le t-s,
		\]
		uniformly in \(\xi\).

	By Plancherel's identity again,
	\[
	\int_{\mathbb R}
	\left|
	\widehat{\mathbf 1}_{[-R,R]}(\xi)
	\right|^2d\xi
	\leq
	C\|\mathbf 1_{[-R,R]}\|_{L^2(\mathbb R)}^2
	\leq
	C R.
	\]
	Thus,
	\[
	\int_0^s
	\left\|
	\varphi_{t-r,R}(\cdot)-\varphi_{s-r,R}(\cdot)
	\right\|_{L^2(\mathbb R)}^2\,dr
	\leq
	C R(t-s).
	\]
	It follows that
	\[
	\|B_{R,1}(s,t)\|_{p}
	\leq
	C R^{1/2}(t-s)^{1/2}.
	\]
	
	Combining the estimates of $B_{R,1}(s,t)$ and $B_{R,2}(s,t)$, we obtain
	\[
	\|G_R(t)-G_R(s)\|_p
	\leq
	C R^{1/2}(t-s)^{1/2}.
	\]
		This proves the desired estimate for $p\geq2$. The case $1\leq p<2$
	then follows from Lyapunov's inequality.
\end{proof}

\textbf{Step 2: Convergence of finite-dimensional distributions.}
\begin{proposition}
	\label{pro:fidi-convergence}	
	For any integer $m\geq1$ and any
	$0\leq t_1<\cdots<t_m\leq T$, we have
	\[
	\left(\frac{1}{\sqrt R}G_R(t_1),\ldots,\frac{1}{\sqrt R}G_R(t_m)\right)
	\xrightarrow{\;\mathrm{d}\;}
	\left(Z(t_1),\ldots,Z(t_m)\right),
	\]
	where $\{Z(t):t\in[0,T]\}$ is a centered Gaussian process with covariance
	function
	\[
	\mathbb E[Z(t)Z(s)]
	=
	2\int_0^{t\wedge s}\xi(r)\,dr.
	\]
	Equivalently,
	\[
	Z(t)=\int_0^t\sqrt{2\xi(r)}\,dB_r.
	\]
\end{proposition}
\begin{proof}
	For \(i=1,\ldots,m\), consider
	\[
	F_R^{(i)}:=\frac{1}{\sqrt R}G_R(t_i)=\frac{1}{\sqrt{R}}\left(\int_{-R}^{R}u(t_i,x)dx-2R\right),
	\]
	and set
\[
F_R:=\left(F_R^{(1)},\ldots,F_R^{(m)}\right).
\]
Then \(F_R^{(i)}=\delta(v_R^{(i)})\), where
	\[
	v_R^{(i)}(s,y)
	=
	\mathbf 1_{[0,t_i]}(s)\frac{\sigma(u(s,y))}{\sqrt R}
	\varphi_{t_i-s,R}(y).
	\]
	Let \(Z=(Z(t_1),\ldots,Z(t_m))\) be the centered Gaussian vector with covariance
	matrix
	\[
	C_{ij}:=\mathbb E[Z(t_i)Z(t_j)]
	=
	2\int_0^{t_i\wedge t_j}\xi(r)\,dr .
	\]
	By the multivariate Malliavin--Stein bound in Proposition \ref{pro:con-for-finite-distri}, for every \(h\in C^2(\mathbb R^m)\)
	with bounded second partial derivatives,
	\[
	\left|\mathbb E h(F_R)-\mathbb E h(Z)\right|
	\le
	\frac12\|h''\|_\infty
	\left(
	\sum_{i,j=1}^m
	\mathbb E\left[
	\left(
	C_{ij}
	-
	\langle DF_R^{(i)},v_R^{(j)}\rangle_{\mathcal H}
	\right)^2
	\right]
	\right)^{1/2}.
	\]
	Therefore, it suffices to prove that, for every \(1\leq i,j\leq m\),
	\[
	\left\langle
	DF_R^{(i)},v_R^{(j)}
	\right\rangle_{\mathcal H}
	\longrightarrow
	C_{ij}
	\quad\text{in }L^2(\Omega).
	\]
	Indeed,
	\[
	\begin{aligned}
		&\mathbb E\left[
		\left(
		C_{ij}
		-
		\left\langle
		DF_R^{(i)},v_R^{(j)}
		\right\rangle_{\mathcal H}
		\right)^2
		\right]\\
		&\quad=
		\left(
		C_{ij}
		-
		\mathbb E\left[
		\left\langle
		DF_R^{(i)},v_R^{(j)}
		\right\rangle_{\mathcal H}
		\right]
		\right)^2+
		\operatorname{Var}\left(
		\left\langle
		DF_R^{(i)},v_R^{(j)}
		\right\rangle_{\mathcal H}
		\right).
	\end{aligned}
	\]
	
	Since \(F_R^{(i)}=\delta(v_R^{(i)})\), the duality relation between
	the Malliavin derivative and the divergence operator gives
	\[
	\begin{aligned}
		\mathbb E\left[
		\left\langle
		DF_R^{(i)},v_R^{(j)}
		\right\rangle_{\mathcal H}
		\right]
		&=
		\mathbb E\left[
		F_R^{(i)}F_R^{(j)}
		\right]\\
		&=
		\frac{1}{R}
		\operatorname{Cov}\bigl(G_R(t_i),G_R(t_j)\bigr).
	\end{aligned}
	\]
	Therefore, by Proposition~\ref{pro:cov_GR},
	\[
	\mathbb E\left[
	\left\langle
	DF_R^{(i)},v_R^{(j)}
	\right\rangle_{\mathcal H}
	\right]
	\longrightarrow
	2\int_0^{t_i\wedge t_j}\xi(r)\,dr
	=
	C_{ij}.
	\]
	
	It remains to prove that
\[
\operatorname{Var}\left(
\left\langle DF_R^{(i)},v_R^{(j)}
\right\rangle_{\mathcal H}
\right)
\longrightarrow0.
\]
For this purpose, we use the same decomposition as in the proof of
Theorem~\ref{QCLT}. More precisely,
	\[
	\langle DF_R^{(i)},v_R^{(j)}\rangle_{\mathcal H}
	=
	I_{1,i,j}(R)+I_{2,i,j}(R),
	\]
	where
	\[
		I_{1,i,j}(R)
		:=
		\frac1R
		\int_0^{t_i\wedge t_j}
		\int_{\mathbb R}
		\int_{-R}^{R}\int_{-R}^{R}
		G_{t_i-s}(x_1-y)G_{t_j-s}( x_2-y) 
		\sigma^2(u(s,y))
		\,dx_1\,dx_2\,dy\,ds,
	\]
	and
	\[
	\begin{aligned}
		I_{2,i,j}(R)
		&:=
		\frac1R
		\int_0^{t_i\wedge t_j}
		\int_{\mathbb R}
		\int_{-R}^{R}\int_{-R}^{R}
		G_{t_j-s}( x_1-y)\sigma(u(s,y))  \\
		&\quad\times
		\left(
		\int_s^{t_i}\int_{\mathbb R}
		G_{t_i-r}(x_2-z)\Sigma(r,z)D_{s,y}u(r,z)\,
		W(dr,dz)
		\right)
		\,dx_1\,dx_2\,dy\,ds .
	\end{aligned}
	\]
Then using \eqref{eq:sqrt-Var-sqrt}, we get
\[
	\sqrt{
		\operatorname{Var}\left(
		\left\langle
		DF_R^{(i)},v_R^{(j)}
		\right\rangle_{\mathcal H}
		\right)
	}
	\leq
	\sqrt{
		\operatorname{Var}\bigl(I_{1,i,j}(R)\bigr)
	}
	+
	\sqrt{
		\operatorname{Var}\bigl(I_{2,i,j}(R)\bigr)
	}.
\]
Applying the same arguments as those used to estimate \(A_1\) and
\(A_2\) in the proof of Theorem~\ref{QCLT} to the mixed-time kernels
appearing above, we obtain
\[
\sqrt{\operatorname{Var}(I_{1,i,j}(R))}
\leq
\frac{C_T}{\sqrt R},
\qquad
\sqrt{\operatorname{Var}(I_{2,i,j}(R))}
\leq
\frac{C_T}{\sqrt R}.
\]
Here \(C_T\) is independent of \(R,i\), and \(j\), provided that
\(t_i,t_j\in[0,T]\).

Hence,
\[
\operatorname{Var}\left(
\left\langle
DF_R^{(i)},v_R^{(j)}
\right\rangle_{\mathcal H}
\right)
\leq
\frac{C_T}{R}
\longrightarrow0.
\]
Consequently,
\[
\mathbb E\left[
\left(
C_{ij}
-
\left\langle
DF_R^{(i)},v_R^{(j)}
\right\rangle_{\mathcal H}
\right)^2
\right]
\longrightarrow0.
\]
This completes the proof.
\end{proof}
By Remark~\ref{rem:nondegenerate-variance}, the limiting variance is strictly
positive for every \(t>0\). 
Therefore, Theorem \ref{FCLT} follows by combining Proposition \ref{pro:time-increment-spatial-average} and Proposition \ref{pro:fidi-convergence}.


\appendix
\section{Proof of Lemma \ref{le:malliavin}}\label{6}
	Fix \((s,y)\) with \(0\le s<T\) and \(y\in\mathbb R\). For \(t>0\) and \(x\in\mathbb R\),
	define
	\[
	g(t,x):=\|D_{s,y}u(s+t,y+x)\|_p^2.
	\]
	Then \eqref{eq:malliavin-iterate-G} can be rewritten as
	\begin{equation}\label{eq:g-recursion}
		g(t,x)
		\le
		C G_t^2(x)
		+
		C\int_0^t\int_{\mathbb R}
		G_{t-r}^2(x-\xi)\,g(r,\xi)\,d\xi\,dr,
		\qquad 0<t\le T-s.
	\end{equation}
	What we need to prove is that
	\[
	g(t,x)\le C\,t^{-\beta/\alpha}G_t(x).
	\]
	
	By Lemma~A.1 in \cite{chen2019}, it suffices to establish the desired estimate
	in the case when \eqref{eq:g-recursion} holds with equality. 
	
	We now introduce the Picard sequence \((g_n)_{n\ge0}\) by
	\[
	g_0(t,x):=C\,G_t^2(x),
	\]
	and, for \(n\ge0\),
	\begin{equation}\label{eq:gn-recursion}
		g_{n+1}(t,x):=
		C\,G_t^2(x)
		+
		C\int_0^t\int_{\mathbb R}
		G_{t-r}^2(x-\xi)\,g_n(r,\xi)\,d\xi\,dr.
	\end{equation}
	Since the kernel \(G\) is nonnegative, \((g_n(t,x))_{n\ge0}\) is nondecreasing
	for every fixed \((t,x)\in(0,T-s]\times\mathbb R\).
	
	We claim that there exists a sequence of nonnegative functions \((A_n)_{n\ge0}\)
	such that
	\begin{equation}\label{eq:gn-majorant}
		g_n(t,x)\le A_n(t)\,G_t(x),
		\qquad n\ge0,
	\end{equation}
	where
	\[
	A_0(t)=Ct^{-\beta/\alpha},
	\]
	and, for \(n\ge0\),
	\begin{equation}\label{eq:An-recursion}
		A_{n+1}(t)=C t^{-\beta/\alpha}
		+
		C\int_0^t (t-r)^{-\beta/\alpha}A_n(r)\,dr.
	\end{equation}
	
	We first verify \eqref{eq:gn-majorant} for \(n=0\). By Lemma~\ref{le:G-sup},
	\[
	G_t(x)\le C t^{-\beta/\alpha},
	\]
	so that
	\[
	g_0(t,x)=C G_t^2(x)\le C t^{-\beta/\alpha}G_t(x).
	\]
	Hence \eqref{eq:gn-majorant} holds with \(A_0(t)=Ct^{-\beta/\alpha}\).
	
	Suppose now that \eqref{eq:gn-majorant} holds for some \(n\ge0\). Then, using
	\eqref{eq:gn-recursion},
	\begin{align*}
		g_{n+1}(t,x)
		&=
		C G_t^2(x)
		+
		C\int_0^t\int_{\mathbb R}
		G_{t-r}^2(x-\xi)\,g_n(r,\xi)\,d\xi\,dr \\
		&\le
		C G_t^2(x)
		+
		C\int_0^t\int_{\mathbb R}
		G_{t-r}^2(x-\xi)\,A_n(r)G_r(\xi)\,d\xi\,dr \\
		&\le
		C t^{-\beta/\alpha}G_t(x)
		+
		C\int_0^t A_n(r)
		\left(
		\int_{\mathbb R}G_{t-r}^2(x-\xi)G_r(\xi)\,d\xi
		\right)dr.
	\end{align*}
	By Corollary~\ref{cor:G2G},
	\[
	\int_{\mathbb R}G_{t-r}^2(x-\xi)G_r(\xi)\,d\xi
	\le C (t-r)^{-\beta/\alpha}G_t(x).
	\]
	Therefore,
	\[
	g_{n+1}(t,x)
	\le
	G_t(x)\left[
	C t^{-\beta/\alpha}
	+
	C\int_0^t (t-r)^{-\beta/\alpha}A_n(r)\,dr
	\right]
	=
	A_{n+1}(t)\,G_t(x),
	\]
	which proves \eqref{eq:gn-majorant}.
	
	It remains to estimate \((A_n)_{n\ge0}\). Set
	\[
	\rho:=\frac{\beta}{\alpha}\in(0,1).
	\]
	Then \eqref{eq:An-recursion} becomes
	\[
	A_{n+1}(t)=C t^{-\rho}+C\int_0^t (t-r)^{-\rho}A_n(r)\,dr.
	\]
	By iterating this relation and using the Beta function identity, one obtains
	\[
	A_n(t)\le
	C t^{-\rho}
	\sum_{k=0}^n
	\frac{\big(C\Gamma(1-\rho)t^{1-\rho}\big)^k}
	{\Gamma((k+1)(1-\rho))}.
	\]
	Equivalently,
	\[
	A_n(t)\le
	C t^{-\rho}
	E_{1-\rho,\,1-\rho}\!\left(C\Gamma(1-\rho)t^{1-\rho}\right),
	\]
	where \(E_{a,b}\) denotes the two-parameter Mittag--Leffler function. Since
	\(0<t\le T\), the quantity
	\[
	E_{1-\rho,\,1-\rho}\!\big(C\Gamma(1-\rho)t^{1-\rho}\big)
	\]
	is uniformly bounded on \([0,T]\). Consequently,
	\[
	A_n(t)\le C_T t^{-\rho},
	\qquad 0<t\le T,
	\]
	with \(C_T\) independent of \(n\).
	
	Combining this with \eqref{eq:gn-majorant}, we obtain
	\[
	g_n(t,x)\le C_T t^{-\rho}G_t(x),
	\qquad 0<t\le T,\ x\in\mathbb R.
	\]
	Letting \(n\to\infty\), we conclude that
	\[
	g(t,x)\le C_T t^{-\rho}G_t(x).
	\]
	Recalling the definition of \(g\), this is exactly
	\[
	\|D_{s,y}u(r,z)\|_p^2
	\le C_T (r-s)^{-\beta/\alpha}G_{r-s}(z-y),
	\]
	which proves \eqref{eq:bound-D-square}. Taking square roots yields
	\eqref{eq:bound-D-half}. This completes the proof.


\begin{thebibliography}{99}
\bibitem{assaad2022fra}
O. Assaad, D. Nualart, C. A. Tudor and L. Viitasaari,
Quantitative normal approximations for the stochastic fractional heat equation,
\emph{Stoch. Partial Differ. Equ. Anal. Comput.} \textbf{10} (2022), 223--254.

\bibitem{balan2022hyperbolic}
R. M. Balan, D. Nualart, L. Quer-Sardanyons and G. Zheng,
The hyperbolic Anderson model: moment estimates of the Malliavin derivatives and applications,
\emph{Stoch. Partial Differ. Equ. Anal. Comput.} \textbf{10} (2022), 757--827.

\bibitem{balan2023central}
R. M. Balan and W. Yuan,
Central limit theorems for heat equation with time-independent noise: The regular and rough cases,
\emph{Infin. Dimens. Anal. Quantum Probab. Relat. Top.} \textbf{26} (2023), 2250029.

\bibitem{balan2025gaussian}
R. M. Balan, J. Huang, X. Wang, P. Xia and W. Yuan,
Gaussian fluctuations for the wave equation under rough random perturbations,
\emph{Stochastic Process. Appl.} \textbf{182} (2025), 104569.

\bibitem{balan2025gaussian2}
R. M. Balan and M. Salins,
Gaussian fluctuations for the nonlinear stochastic heat equation with drift,
\emph{arXiv preprint arXiv:2512.12119} (2025).

\bibitem{bolanos2021averaging}
R. Bola\~nos Guerrero, D. Nualart and G. Zheng,
Averaging 2d stochastic wave equation,
\emph{Electron. J. Probab.} \textbf{26} (2021), 1--32.

\bibitem{caputo1967linear}
M. Caputo,
Linear models of dissipation whose \(Q\) is almost frequency independent---II,
\emph{Geophys. J. Int.} \textbf{13} (1967), 529--539.

\bibitem{chen2015moments}
L. Chen and R. C. Dalang,
Moments, intermittency and growth indices for the nonlinear fractional stochastic heat equation,
\emph{Stoch. Partial Differ. Equ. Anal. Comput.} \textbf{3} (2015), 360--397.

\bibitem{chen2019}
L. Chen and J. Huang,
Comparison principle for stochastic heat equation on $\mathbb{R}^d$,
\emph{Ann. Probab.} \textbf{47} (2019), 989--1035.

\bibitem{chen2021spatial}
L. Chen, D. Khoshnevisan, D. Nualart and F. Pu,
Spatial ergodicity for SPDEs via Poincaré-type inequalities,
\emph{Electron. J. Probab.} \textbf{26} (2021), 1--37.

\bibitem{chen2022central}
L. Chen, D. Khoshnevisan, D. Nualart and F. Pu,
Central limit theorems for parabolic stochastic partial differential equations,
\emph{Ann. Inst. Henri Poincar\'e Probab. Stat.} \textbf{58} (2022), 1052--1077.

\bibitem{chen2022spatial}
L. Chen, D. Khoshnevisan, D. Nualart and F. Pu,
Spatial ergodicity and central limit theorems for parabolic Anderson model with delta initial condition,
\emph{J. Funct. Anal.} \textbf{282} (2022), 109290.

\bibitem{chen2023central}
L. Chen, D. Khoshnevisan, D. Nualart and F. Pu,
Central limit theorems for spatial averages of the stochastic heat equation via Malliavin--Stein's method,
\emph{Stoch. Partial Differ. Equ. Anal. Comput.} \textbf{11} (2023), 122--176.

\bibitem{dalang1999extending}
R. C. Dalang,
Extending martingale measure stochastic integral with applications to spatially homogeneous SPDE's,
\emph{Electron. Commun. Probab.} \textbf{4} (1999), 43--61.

\bibitem{debbi2005solutions}
L. Debbi and M. Dozzi,
On the solutions of nonlinear stochastic fractional partial differential equations in one spatial dimension,
\emph{Stochastic Process. Appl.} \textbf{115} (2005), 1764--1781.

\bibitem{dhoyer2022spatial}
R. Dhoyer and C. A. Tudor,
Spatial average for the solution to the heat equation with Rosenblatt noise,
\emph{Stochastic Anal. Appl.} \textbf{40} (2022), 951--966.

\bibitem{foondun2017asymptotic}
M. Foondun and E. Nane,
Asymptotic properties of some space-time fractional stochastic equations,
\emph{Math. Z.} \textbf{287} (2017), 493--519.

\bibitem{huang2020}
J. Huang, D. Nualart and L. Viitasaari,
A central limit theorem for the stochastic heat equation,
\emph{Stochastic Process. Appl.} \textbf{130} (2020), 7170--7184.

\bibitem{huang2020color}
J. Huang, D. Nualart, L. Viitasaari and G. Zheng,
Gaussian fluctuations for the stochastic heat equation with colored noise,
\emph{Stoch. Partial Differ. Equ. Anal. Comput.} \textbf{8} (2020), 402--421.

\bibitem{khoshnevisan2021spatial}
D. Khoshnevisan, D. Nualart and F. Pu,
Spatial Stationarity, Ergodicity, and CLT for Parabolic Anderson Model with Delta Initial Condition in Dimension $d\ge$1,
\emph{SIAM J. Math. Anal.} \textbf{53} (2021), 2084--2133.

\bibitem{kim2022limit}
K. Kim and J. Yi,
Limit theorems for time-dependent averages of nonlinear stochastic heat equations,
\emph{Bernoulli} \textbf{28} (2022), 214--238.

\bibitem{li2023gaussian}
Z. Li and F. Pu,
Gaussian fluctuation for spatial average of super-Brownian motion,
\emph{Stochastic Anal. Appl.} \textbf{41} (2023), 752--769.

\bibitem{liu2023gaussian}
J. Liu and G. Shen,
Gaussian fluctuation for spatial average of the stochastic pseudo-partial differential equation with fractional noise,
\emph{ALEA Lat. Am. J. Probab. Math. Stat.} \textbf{20} (2023), 1483--1509.


\bibitem{meerschaert2004limit}
M. M. Meerschaert and H. P. Scheffler,
Limit theorems for continuous-time random walks with infinite mean waiting times,
\emph{J. Appl. Probab.} \textbf{41} (2004), 623--638.

\bibitem{mijena2015space}
J. B. Mijena and E. Nane,
Space–time fractional stochastic partial differential equations,
\emph{Stochastic Process. Appl.} \textbf{125} (2015), 3301--3326.


\bibitem{nourdin2012normal}
I. Nourdin and G. Peccati,
\emph{Normal Approximations with Malliavin Calculus: From Stein's Method to Universality},
Cambridge Tracts in Mathematics, Vol. 192, Cambridge University Press, Cambridge, 2012.

\bibitem{Nualart2006}
D. Nualart,
\emph{The Malliavin Calculus and Related Topics},
Second Edition, Springer-Verlag, Berlin, 2006.

\bibitem{nualart2021spatial}
D. Nualart, X. Song and G. Zheng,
Spatial averages for the parabolic Anderson model driven by rough noise,
\emph{ALEA Lat. Am. J. Probab. Math. Stat.} \textbf{18} (2021).

\bibitem{nualart2022quantitative}
D. Nualart, P. Xia and G. Zheng,
Quantitative central limit theorems for the parabolic Anderson model driven by colored noises,
\emph{Electron. J. Probab.} \textbf{27} (2022), 1--43.

\bibitem{nualart2022central}
D. Nualart and G. Zheng,
Central limit theorems for stochastic wave equations in dimensions one and two,
\emph{Stoch. Partial Differ. Equ. Anal. Comput.} \textbf{10} (2022), 392--418.

\bibitem{pu2022gaussian}
F. Pu,
Gaussian fluctuation for spatial average of parabolic Anderson model with Neumann/Dirichlet/periodic boundary conditions,
\emph{Trans. Amer. Math. Soc.} \textbf{375} (2022), 2481--2509.


\bibitem{Walsh1986}
J. B. Walsh,
An introduction to stochastic partial differential equations,
in: \emph{École d'Été de Probabilités de Saint-Flour XIV---1984}, Lecture Notes in Math., Vol. 1180, Springer, Berlin, 1986, pp. 265--439.

\bibitem{yan2018large}
L. Yan and X. Yin,
Large deviation principle for a space-time fractional stochastic heat equation with fractional noise,
\emph{Fract. Calc. Appl. Anal.} \textbf{21} (2018), 462--485.

\bibitem{zhang2025functional}
W. Zhang, Y. Zhang and J. Li,
Functional central limit theorems for spatial averages of the parabolic Anderson model with delta initial condition in dimension $d\ge$1,
\emph{Lithuanian Math. J.} \textbf{65} (2025), 608--642.
\end{thebibliography}
\end{document}